\documentclass[pdflatex,sn-mathphys-num]{sn-jnl}

\usepackage{graphicx}%
\usepackage{multirow}%
\usepackage{amsmath,amssymb,amsfonts}%
\usepackage{amsthm}%
\usepackage[title]{appendix}%
\usepackage{xcolor}%
\usepackage{textcomp}%
\usepackage{manyfoot}%
\usepackage{booktabs}%
\usepackage{algorithm}%
\usepackage{algorithmicx}%
\usepackage{algpseudocode}%
\usepackage{listings}%
\usepackage{framed}
\usepackage{nameref}
\usepackage{orcidlink}
\usepackage{amsmath, amsthm, amssymb, amsfonts, mathrsfs}
\usepackage{enumerate}
\usepackage{geometry}
\usepackage{mathtools}
\usepackage{hyperref}
\usepackage{cleveref}

\usepackage{ esint }

\theoremstyle{plain}
\newtheorem{theorem}{Theorem}[section]
\newtheorem{proposition}[theorem]{Proposition}
\newtheorem{lemma}[theorem]{Lemma}
\newtheorem{corollary}[theorem]{Corollary}
\theoremstyle{definition}
\newtheorem{definition}[theorem]{Definition}
\newtheorem{rem}[theorem]{Remark}
\newtheorem{example}[theorem]{Example}

\DeclareMathOperator{\ext}{\textup{ext}}

\DeclareMathOperator{\inter}{\textup{int}}
\DeclareMathOperator{\LO}{\textup{span}}
\DeclareMathOperator{\relint}{\textup{relint}}

\DeclareMathOperator{\rank}{rank}

\DeclareMathOperator{\dist}{\textup{dist}}
\DeclareMathOperator{\conv}{\textup{conv}}
\DeclareMathOperator{\aff}{\textup{aff}}
\DeclareMathOperator{\diam}{\textup{diam}}
\newcommand{\dH}[1]{\,\mathrm d\mathcal{H}^{#1}}
\newcommand{\dlambda}[1]{\,\mathrm d\lambda^{#1}}

\begin{document}

\title[Anisotropic Minkowski Content for Countably $\mathcal{H}^k$-rectifiable Sets]{Anisotropic Minkowski Content for Countably $\mathcal{H}^k$-rectifiable Sets}

\author*[]{\fnm{Filip} \sur{Fryš}\,\orcidlink{0009-0000-2134-0109}}\email{filip.frys@matfyz.cuni.cz}

\affil[]{ \orgname{Charles University, Faculty of Mathematics and Physics}, \orgaddress{\street{Sokolovsk\'{a} 83}, \city{Prague 8}, \postcode{186 75}, \state{Czech Republic}}}

\abstract{
We study anisotropic Minkowski content for lower-dimensional rectifiable sets. First, we prove that, for every convex body \(C\subseteq\mathbb R^n\), the \(k\)-dimensional \(C\)-anisotropic Minkowski content of every compact \(k\)-rectifiable set exists. We show that it is given by an integral involving the \((n-k)\)-dimensional volumes of the projections of \(C\) onto the approximate normal spaces of the set.

We then establish the same formula for closed countably \(\mathcal H^k\)-rectifiable sets of finite \(\mathcal H^k\)-measure satisfying an AFP-\(k\)-condition relative to the linear span of \(C\), provided that the condition is witnessed by a finite Radon measure.

Finally, we prove that, for a countably \(\mathcal H^k\)-rectifiable set, if the formula holds for one full-dimensional convex body, then it holds for every full-dimensional convex body.
}


\maketitle\section*{Acknowledgments}
I would like to express my gratitude to my supervisor Jan Rataj for his support and valuable advice. I am also grateful to the anonymous reviewer for a careful reading of the manuscript and for several helpful comments and suggestions.
I gratefully acknowledge the financial support provided by Charles University under grants
PRIMUS/24/SCI/009 and GAUK 34126.

\section*{Introduction}\label{secIntro}

This paper investigates anisotropic versions of lower-dimensional \emph{Minkowski content} and establishes a generalization of Federer's classical isotropic result for $k$-rectifiable compact sets \cite[Theorem~3.2.39]{federer}.

Let $S\subseteq\mathbb{R}^n$ be a nonempty compact set. For $r>0$, the \emph{(isotropic) tubular $r$-neighborhood} of $S$ is defined by
\[
S_r \coloneq \{x\in\mathbb{R}^n : \dist(x,S)\le r\}.
\]
Equivalently, $S_r$ can be written as the Minkowski sum
\[
S_r = S \oplus B(0,r) \coloneq \bigcup_{s\in S} B(s,r).
\]

The volume $\lambda^n(S_r)$ and its asymptotic behavior as $r\to0_+$ have been studied extensively. To capture this behavior, one introduces the \emph{$k$-dimensional Minkowski content} of $S$,
\[
\mathcal{M}^k_{B(0,1)}(S) \coloneq 
\lim_{r\to0_+}\frac{\lambda^n(S_r)}{\omega_{n-k} r^{n-k}},
\]
whenever the limit exists, where $\omega_{n-k}$ denotes the volume of the $(n-k)$-dimensional Euclidean unit ball.

For certain classes of sets, the volume $\lambda^n(S_r)$ admits a polynomial expansion in $r$. If $K$ is a convex body, that is, a nonempty compact convex set, then $$\lambda^n(K_r)=\sum_{j=0}^n{n \choose j} V_n\big(K[j], B(0, 1)[n-j]\big)r^{n-j}, \qquad r>0,$$ where the coefficients $V_n\big(K[j], B(0, 1)[n-j]\big)$, $j\in \{0, \dots, n\}$, are the mixed volumes of $j$ copies of $K$ and $n-j$ copies of $B(0, 1)$ (see, for instance, \cite[Chapter~5]{rolf}). 

Similarly, if $S$ is a compact smooth $k$-dimensional submanifold without
boundary embedded in $\mathbb{R}^n$, Weyl's tube formula gives a polynomial
expansion of $\lambda^n(S_r)$ in powers of $r$ for all sufficiently small
$r>0$; see \cite{We}. 

In both cases, the above limit therefore exists.

\medskip

In a fundamental result, Federer \cite[Theorem~3.2.39]{federer} proved that if $S$ is compact and $k$-rectifiable, then $\mathcal{M}^k_{B(0,1)}(S)$ exists and coincides with the $k$-dimensional Hausdorff measure $\mathcal{H}^k(S)$.

This result was later extended by Ambrosio, Fusco, and Pallara \cite[Theorem~2.104]{ambrosio} to countably $\mathcal{H}^k$-rectifiable compact sets. More precisely, if $S$ is countably $\mathcal{H}^k$-rectifiable and compact, and there exist $\gamma>0$ and a probability Radon measure $\mu$, absolutely continuous with respect to $\mathcal{H}^k$, such that
\begin{equation}\label{Eq30}
\mu\big(B(x,r)\big) \geq \gamma r^k, \qquad x\in S,\ r\in(0,1),
\end{equation}
then the limit $\mathcal{M}^k_{B(0,1)}(S)$ exists and equals $\mathcal{H}^k(S)$.

The aim of this work is to leave the isotropic framework and study an \emph{anisotropic} counterpart of these results. Given a convex body $C\subseteq\mathbb{R}^n$ with $0\in\relint(C)$, we define the \emph{$C$-anisotropic tubular neighborhood} of $S$ by

$$S \oplus rC \coloneq \bigcup_{s\in S} (s+rC), \qquad r>0.$$

Accordingly, the \emph{$k$-dimensional $C$-anisotropic Minkowski content} of $S$ is defined as

$$\mathcal{M}^k_C(S) \coloneq 
\lim_{r\to0_+}\frac{\lambda^n(S\oplus rC)}{\omega_{n-k} r^{n-k}},$$ whenever the limit exists.

So far, the existence of $\mathcal{M}^k_C(S)$ has been well understood only in the case $k=n-1$ and for convex bodies $C$ containing the origin in their interior. In this setting, Lussardi and Villa \cite[Theorem 3.4]{Villa2} proved that condition \eqref{Eq30} guarantees the existence of $\mathcal{M}^{n-1}_C(S)$ for countably $\mathcal{H}^{n-1}$-rectifiable compact sets $S$, and that

$$\mathcal{M}^{n-1}_C(S)
=
\frac12 \int\limits_S \bigl(h_C(\nu_S)+h_C(-\nu_S)\bigr)\,\dH{n-1},$$
where
$$h_C(x)\coloneq \max_{y\in C} x\cdot y$$
denotes the support function of $C$, and $\nu_S(x)$ is a unit normal vector to the approximate tangent space of $S$ at $x$ for $\mathcal{H}^{n-1}$-almost every $x\in S$.

More recently, Kiderlen and Rataj \cite[Theorem 8]{Rataj5} extended this result to lower-dimensional structuring elements. They introduced a condition analogous to \eqref{Eq30}: if $S$ is a countably $\mathcal{H}^{n-1}$-rectifiable compact set with $\mathcal{H}^{n-1}(S)<\infty$, $L$ is an $m$-dimensional linear subspace of $\mathbb{R}^n$, and there exist $\gamma>0$ and a probability Radon measure $\mu$, absolutely continuous with respect to $\mathcal{H}^{n-1}$, such that
\begin{equation}\label{Eq31}
\mu\big(B(x,r)\big) \geq \gamma r^{m-1}\mathcal{H}^{n-m}\Big(P_{L^\perp}\big(S\cap B(x,r)\big)\Big),
\qquad x\in S,\ r\in(0,1),
\end{equation}
then for any convex body $C\subseteq L$ the limit $\mathcal{M}^{n-1}_C(S)$ exists and admits the same integral representation. Here, $P_{L^{\perp}}$ stands for the orthogonal projection onto the linear subspace $L^{\perp}$.

\medskip

In this paper we show that, for every $k\in\{1, \dots, n-1\}$ and every convex body $C$, the $C$-anisotropic $k$-dimensional Minkowski content of a $k$-rectifiable compact set $S$ exists and is given by
\begin{equation}\label{Eq40}
\mathcal{M}^k_C(S)
=
\frac{1}{\omega_{n-k}}
\int\limits_S \mathcal{H}^{n-k}\big(P_{(\textup{T}^k_xS)^{\perp}}(C)\big)\dH{k}(x),
\end{equation}
where $\textup{T}^k_xS$ denotes the \emph{approximate tangent space to $S$ at $x$.}

We further introduce a condition extending \eqref{Eq30} and \eqref{Eq31} that ensures the existence of $\mathcal{M}^k_C(S)$ for countably $\mathcal{H}^k$-rectifiable closed sets $S$ with $\mathcal{H}^k(S)<\infty$, and arbitrary convex bodies $C$. Namely, we show that if there exist a finite Radon measure $\mu$, absolutely continuous with respect to $\mathcal{H}^k$, and a constant $\gamma>0$ such that $$\mu\big(B(x, r)\big)\geq\gamma r^{k+m-n}\mathcal{H}^{n-m}\Big(P_{L^{\perp}}\big(S\cap B(x, r)\big)\Big),\quad\text{$x\in S$, $r\in (0, 1)$,}$$ where $L$ is the $m$-dimensional subspace such that $\LO(C)=L$, then the limit $\mathcal{M}^k_C(S)$ exists and coincides with the integral $$\frac{1}{\omega_{n-k}}
\int\limits_S \mathcal{H}^{n-k}\big(P_{(\textup{T}^k_xS)^{\perp}}(C)\big)\dH{k}(x).$$

Finally, we show that if \eqref{Eq40} holds for some convex body with nonempty interior, then it holds for all convex bodies with nonempty interior.

\medskip

The rest of the paper unfolds as follows.

Section \ref{secPrelim} introduces the notation used throughout and gathers the essential background from geometric measure theory and convex geometry. In particular, we recall the basic properties of convex bodies and mixed volumes.

In Section \ref{secDef}, we define the $s$-dimensional anisotropic Minkowski content. We first establish the result for convex bodies, which serves as a guiding example, and then extend it to sets that arise as compact subsets of $C^{1}$ $k$-graphs.

Section \ref{count} addresses the general rectifiable case. Here we show that the $k$-dimensional $C$-anisotropic Minkowski content of a countably $\mathcal{H}^{k}$-rectifiable closed set with finite $\mathcal{H}^k$-measure exists and attains the expected value, provided an appropriate form of the AFP-condition is satisfied.

Finally, in Section \ref{secFurther}, we prove that if the $k$-dimensional $C$-anisotropic Minkowski content of a countably $\mathcal{H}^k$-rectifiable set $S$ exists and takes the correct value for one full-dimensional convex body $C$, then the same conclusion holds for every full-dimensional convex body.

\section{Preliminaries}
\label{secPrelim}
\subsection{Notation}
 Let $n\in\mathbb{N}$ (and so $n\geq1$). For $x, y\in\mathbb{R}^n$, we denote by $x\cdot y=\sum_{i=1}^{n}x_i y_i$ the \textit{Euclidean inner product of $x$ and $y$}, and the symbol $|\cdot|$ stands for the norm induced by this inner product. By $\Vert\cdot\Vert$ we denote the standard operator norm on the space of all linear endomorphisms of $\mathbb{R}^n$. 
 
 We let $\mathbb{N}_0\coloneqq\mathbb{N}\cup\{0\}.$

If $A, B\subseteq\mathbb{R}^n$, then: \begin{itemize}
    \item The symbol $A\oplus B\coloneq \{a+b: (a, b)\in A\times B\}$ stands for the \textit{Minkowski sum of $A$ and $B$}. 
    \item The \textit{$r$-multiple of $A$} is denoted by $rA\coloneq\{ra;\;a\in A\}$, where $r\in\mathbb{R}$.
    \item The \textit{symmetric difference of $A$ and $B$} is denoted by $A\Delta B\coloneq (A\setminus B)\cup(B\setminus A)$.
    \item The \textit{orthogonal complement of $A$} will be denoted by  $A^{\perp}$, the \textit{linear span of $A$} by $\LO (A)$ and the \textit{affine span of $A$} by $\aff(A)$. If $A$ is a linear subspace of $\mathbb{R}^n$, we denote by $\dim(A)$ the \emph{dimension of $A$.}
    \item We denote the \textit{interior of $A$} by $\inter (A),$ the \textit{exterior of $A$} by $\ext (A)$, the \emph{closure of $A$} by $\overline{A}$, and the \textit{boundary of $A$} 
    by $\partial A.$ If $A$ is a subset of a metric space $X$, we denote its interior and boundary relative to $X$ by
$\inter_X(A)$ and
$\partial_X A$, respectively. Further, we define $\relint(A)\coloneq\inter_{\aff(A)}(A).$
    \item If $A\neq\emptyset$, then $\dist(\cdot, A)$ denotes the classical \textit{Euclidean distance to $A$} and $\diam(A)$ denotes the \textit{diameter of $A$.}
    \item The symbol $\lambda^{n}$ denotes the \textit{$n$-dimensional outer Lebesgue measure}; we use the same notation for its restriction to the Lebesgue measurable sets. Further, $\mathcal{H}^{s}$ denotes the \textit{$s$-dimensional Hausdorff measure}.
    \item The \textit{characteristic function of $A$} will be denoted by  $\chi_A$.

\end{itemize}

Further, let $B(x, r)\equiv B_{\mathbb{R}^n}(x, r)$ denote the closed ball centered at $x$ with radius $r>0$. For $s\geq0$, we write
$$\omega_s\coloneq \frac{\pi^{s/2}}{\Gamma\left(\frac{s}{2}+1\right)},$$ where $\Gamma$ stands for the standard gamma function. For $s\in\mathbb{N}_0$, this is the volume of the unit ball in
$\mathbb R^s$. 

We set $\mathbb{S}^{n-1}\coloneq\partial B(0, 1).$ The \emph{cube centered at $x\in\mathbb{R}^n$ with side length $2s>0$} is defined as $Q_{{\mathbb{R}^n}}(x, s)\coloneq x+[-s, s]^n$. 

The set of all linear subspaces of
$\mathbb{R}^n$ of dimension $k$ will be denoted by $G_k(\mathbb{R}^n).$ If $L\in G_k(\mathbb{R}^n)$, we denote by $B_L(x, r)$ the set $B(x, r)\cap L$ and by $P_{L}$ the orthogonal projection onto $L$.

The \emph{determinant} will be denoted by $\det$.

 Let $\mu$ be a nonnegative measure and let $B$ be a $\mu$-measurable set. The symbol $\mu\big|_{B}$ stands for the \textit{restriction of $\mu$ to $B$.} We abbreviate \textit{almost everywhere} by a.e. and \textit{almost all} by a.a. 

 The space $\mathcal{C}^{\infty}_c(\Omega)$ consists of all infinitely differentiable compactly supported functions on an open set  $\Omega\subseteq\mathbb{R}^n$. The space $\mathcal{C}(A)$ consists of all continuous  functions from $A\subseteq\mathbb{R}^n$ to $\mathbb{R}$. We say that a function $f\colon A\to\mathbb{R}^l$, where $A\subseteq\mathbb{R}^n$, is \emph{of class $\mathcal{C}^k$} provided that there exist $\Omega\supseteq A$ open and a function $g\in\mathcal{C}^k(\Omega; \mathbb{R}^l)$, that is, a $k$-times continuously differentiable function on $\Omega$ to $\mathbb{R}^l$, such that the restriction of $g$ to $A$---$g|_{A}$ is equal to $f$.

    The symbol $D$ will be used for the \textit{total derivative} (and we will identify the derivative and the corresponding matrix that represents the derivative).

    If $\mathcal{V}$ is a collection of subsets of $\mathbb{R}^n$, we define $$\bigcup\mathcal{V}\coloneq\bigcup\{V:V\in\mathcal{V}\}.$$

 \subsection{Geometric Measure Theory}
In this subsection, we formulate some useful results from geometric measure theory. Let us start with the following definition. 
\begin{definition}
        Let $k, n\in\mathbb{N}$ satisfy $1\leq k\leq n$, $f\colon\mathbb{R}^k\to\mathbb{R}^n$ and 
        $g\colon\mathbb{R}^n\to\mathbb{R}^k$ be Lipschitz. We define (almost everywhere) the \textit{Jacobian determinant of $f$ at $x\in\mathbb{R}^k$} and
        the \textit{Jacobian determinant of $g$ at $y\in\mathbb{R}^n$} as \begin{align*}
            J_kf(x)\coloneq \sqrt{\det\big([Df(x)]^{\top}[D f(x)]\big)}&&\text{and}&&J^kg(y)\coloneq \sqrt{\det\big([Dg(y)][Dg(y)]^{\top}\big)}.
        \end{align*}
\end{definition}

\begin{theorem}[Area Formula, \protect{\cite[Theorem 3.2.3]{federer}}] Let  $f\colon\mathbb{R}^k\to\mathbb{R}^n$ be Lipschitz and $E\subseteq\mathbb{R}^k$ be $\mathcal{H}^k$-measurable, where $k, n\in\mathbb{N}$ satisfy $1\leq k\leq n$. Then the map $y\mapsto \mathcal{H}^{0}\big(f^{-1}(\{y\})\cap E\big)$ is $\mathcal{H}^k$-measurable and $$\int\limits_{E}J_kf(x)\dH{k}(x)=\int\limits_{\mathbb{R}^n}\mathcal{H}^{0}\big(f^{-1}(\{y\})\cap E\big)\dH{k}(y).$$
    
\end{theorem}
\begin{theorem}[Coarea Formula, \protect{\cite[Theorem 3.2.11]{federer}}]  Let  $f\colon\mathbb{R}^n\to\mathbb{R}^k$ be Lipschitz and $E\subseteq\mathbb{R}^n$ be $\mathcal{H}^n$-measurable, where $k, n\in\mathbb{N}$ satisfy $1\leq k\leq n$. Then the map $y\mapsto \mathcal{H}^{n-k}\big(f^{-1}(\{y\})\cap E\big)$ is $\mathcal{H}^k$-measurable and $$\int\limits_{E}J^kf(x)\dH{n}(x)=\int\limits_{\mathbb{R}^k}\mathcal{H}^{n-k}\big(f^{-1}(\{y\})\cap E\big)\dH{k}(y).$$
\end{theorem}

The formul\ae in the previous two theorems can be naturally extended for nonnegative measurable functions.

\begin{definition}
        Let $k\in\mathbb{N}_0$ with $k\leq n$ and $S\subseteq\mathbb{R}^n$ be $\mathcal{H}^k$-measurable. 
        
        The set $S$ is said to be \textit{$k$-rectifiable} provided that there exists a bounded set $F\subseteq\mathbb{R}^k$ and a Lipschitz map $f\colon F\to\mathbb{R}^n$ such that $f(F)=S.$
        
        The set $S$ is said to be \textit{countably $\mathcal{H}^{k}$-rectifiable} if there exists a sequence $\{f_i\}_{i=0}^{\infty}$ of Lipschitz functions from $\mathbb{R}^k$ to $\mathbb{R}^n$ such that $$\mathcal{H}^k\left(S\setminus \bigcup_{i\in\mathbb{N}_0}f_i\left(\mathbb{R}^k\right)\right)=0.$$
 \end{definition}

The standard extension theorem for Lipschitz maps (see, for instance, \cite[Theorem~2.10.43]{federer}) implies that every $k$-rectifiable set is, in particular, countably $\mathcal{H}^k$-rectifiable. 

\begin{theorem}[\protect{\cite[Theorem 3.2.22]{federer}}]\label{rezy}
    If $W$ is a countably $\mathcal{H}^{m}$-rectifiable set and $Z$ is a countably $\mathcal{H}^{l}$-rectifiable set, where $m\geq l$, and $f\colon W\to Z$ is Lipschitzian, then for $\mathcal{H}^l$-a.a. $y\in Z$, the set $f^{-1}(\{y\})$ is countably $\mathcal{H}^{m-l}$-rectifiable.
\end{theorem}
We will need the following coarea formula for orthogonal projections.

\begin{rem}[\protect{\cite[Theorem~1.21 and Example~1.22]{zahle}}]
Let $L\subseteq\mathbb{R}^n$ be a linear subspace of dimension $m$ and let
$S\subseteq\mathbb{R}^n$ be a countably $\mathcal{H}^k$-rectifiable set with
$k+m\geq n$. Then
\[
\int\limits_{S} J^{n-m}\big(P_{L^{\perp}}|_{\textup{T}^k_w S}\big)\,\dH{k}(w)
=
\int\limits_{L^{\perp}}
\mathcal{H}^{k+m-n}\big(S\cap(z+L)\big)\,\dH{n-m}(z).
\] This formula can be naturally extended to nonnegative measurable functions.

Here, $$J^{n-m}\big(P_{L^{\perp}}|_{\textup{T}^k_w S}\big)=\sqrt{\det\big([P_{L^{\perp}}|_{\textup{T}^k_w S}][P_{L^{\perp}}|_{\textup{T}^k_w S}]^{\top}\big)}.$$
Moreover,
\[
J^{n-m}\big(P_{L^{\perp}}|_{\textup{T}^k_w S}\big)
=J^{n-k}\big(P_{(\textup{T}^{k}_w S)^{\perp}}|_L\big)\coloneq \sqrt{\det\big([P_{(\textup{T}^{k}_w S)^{\perp}}|_L][P_{(\textup{T}^{k}_w S)^{\perp}}|_L]^{\top}\big)}
.
\]

\end{rem}
 
 \begin{definition}
Let $S \subseteq \mathbb{R}^n$. We say that $S$ is a \emph{Lipschitz $k$-graph} (respectively, a $\mathcal{C}^1$ $k$-graph) if, up to a rotation of $\mathbb{R}^n$, the set $S$ can be written as the graph of a Lipschitz function from $\mathbb{R}^k$ to $\mathbb{R}^{n-k}$ (respectively, of a function of class $\mathcal{C}^1$).
\end{definition}

\begin{rem}[\protect{\cite[p. 80]{ambrosio}}]\label{remark}
    Given a countably $\mathcal{H}^{k}$-rectifiable set $S\subseteq\mathbb{R}^n$, there exists a countable family of pairwise disjoint compact subsets of $S$ that covers $S$ up to an $\mathcal{H}^{k}$-negligible set and each element of which is a subset of a $\mathcal{C}^1$ $k$-graph.
\end{rem}

We use such decompositions to define approximate tangent spaces for
countably $\mathcal{H}^k$-rectifiable sets of possibly infinite
$\mathcal{H}^k$-measure.

\begin{theorem}[\protect{\cite[Remark 2.80 and Theorem 2.83]{ambrosio}}]\label{klim}
        Let $S\subseteq\mathbb{R}^n$ be a countably $\mathcal{H}^k$-rectifiable set of finite $\mathcal{H}^k$-measure. Then for $\mathcal{H}^k$-a.a. $x\in S$, there exists a $k$-dimensional approximate tangent space at $x$ $\textup{T}^k_xS\in G_k(\mathbb{R}^n)$ uniquely determined by the following: \[
\lim_{r\to0_+}\frac{1}{r^k}\int\limits_S \phi\!\left(\frac{z-x}{r}\right)\,\mathrm{d}\mathcal{H}^k(z)
= \int\limits_{\textup{T}_x^kS} \phi(y)\,\mathrm{d}\mathcal{H}^k(y), \qquad \phi\in\mathcal{C}^{\infty}_c(\mathbb{R}^n).
\]
\end{theorem}

\begin{definition}[\protect{\cite[Definition 2.86]{ambrosio}}]
     Let $S\subseteq\mathbb{R}^n$ be a countably $\mathcal{H}^k$-rectifiable set, and let $\{S_i\}_{i=1}^{\infty}$ be a partition of $\mathcal{H}^k$-a.a. of $S$ into countably $\mathcal{H}^k$-rectifiable sets of finite $\mathcal{H}^k$-measure. For $\mathcal{H}^k$-a.e. $x\in S_i$, we define the \emph{approximate tangent space of $S$ at $x$}, written $\textup{T}_x^kS$, as $$\textup{T}_x^kS\coloneqq \textup{T}_x^kS_i,$$ where the latter is defined in Theorem \ref{klim}.
     
     In particular, if $k=n-1$, for $\mathcal{H}^{n-1}$-a.a. $x\in S$, there exists a unit vector $\nu_S(x)$, unique up to sign, that generates the linear subspace $(\textup{T}^{n-1}_xS)^{\perp}$.
\end{definition}
The following compatibility property shows that this definition is
independent, up to an $\mathcal{H}^k$-negligible set, of the chosen
decomposition.
\begin{rem}[\protect{\cite[Remark 2.87]{ambrosio}}]\label{consistency}
    If $S^{\prime}, S\subseteq\mathbb{R}^n$ are countably $\mathcal{H}^k$-rectifiable sets satisfying $S^{\prime}\subseteq S$, then $$\textup{T}^k_xS=\textup{T}_x^kS^{\prime}\qquad\text{for $\mathcal{H}^k$-a.a. $x\in S^{\prime}.$}$$ Furthermore, the approximate tangent-space map
\[
x\mapsto \textup{T}_x^kS, \qquad x\in S,
\]
is \(\mathcal H^k\)-measurable and is defined uniquely up to changes on an \(\mathcal H^k\)-negligible subset of \(S\).
\end{rem}

We now prove the following covering lemma suggested together with proof by the reviewer, which is a weaker version of the
Besicovitch covering theorem but is sufficient for our purposes.
\begin{lemma}\label{bes}
       Let $A\subseteq\mathbb{R}^n$, and let $\rho>0$. Then there exists an at most countable set $S\subseteq A$ such that \begin{align*}
          A\subseteq\bigcup_{x\in S}B(x, \rho) &&\text{and}&& \sum_{x\in S}\chi_{B(x, \rho)}\leq 3^n.
       \end{align*}
\end{lemma}
\begin{proof}
    Let $S\subseteq A$ be a strictly $\rho$-separated set (that is, whenever $x, y\in S$ with $x\neq y$, then $|x-y|>\rho$), maximal with respect to inclusion. Then $S$ is at most countable since $\mathbb{R}^n$ is separable. Maximality ensures that for every $y\in A$ there exists some $x\in S$ such that $|x-y|\leq \rho.$ Consequently, $$A\subseteq\bigcup_{x\in S}B(x, \rho).$$ Further, let $$S_{y}\coloneqq \{x\in S : y\in B(x, \rho)\}, \qquad y\in\mathbb{R}^n.$$ Let $y\in \mathbb{R}^n$. The balls $\{B(x, \rho/2)\}_{x\in S_{y}}$ are pairwise disjoint and contained in $B(y, 3\rho/2)$, which in turn implies $$\mathcal{H}^0(S_y)\omega_n(\rho/2)^n\leq \lambda^n\Big(\bigcup_{x\in S_y} B(x, \rho/2)\Big)\leq \omega_n(3\rho/2)^n.$$ Thus $\mathcal{H}^0(S_y)\leq 3^n.$ Finally, for any $y\in\mathbb{R}^n$ we have $$\sum_{x\in S}\chi_{B(x,\rho)}(y)=\mathcal{H}^0(S_y)\leq 3^n.$$
\end{proof}
\subsection{Convex Bodies}
In this subsection, we summarize the basic facts about convex bodies and their properties. For a more detailed exposition, we refer the reader to~\cite{rolf}.
\begin{definition}\label{convexity}
    A \emph{convex body} is a nonempty compact convex
subset of $\mathbb R^n$. We denote by $\mathcal C^n$ the family of all convex
bodies $C$ satisfying $0\in\relint(C)$. For a convex body $C$, the \emph{dimension}
$\dim(C)$ is the dimension of $\aff(C)$. 
    
    Further, we set \begin{align*}
        \mathcal{C}^{n, k}\coloneq\{C\in\mathcal{C}^n : \dim(C)=k\}, \qquad k\in\{0, \dots, n\}.
    \end{align*}
    \end{definition}

    In this paper, we work mainly with convex bodies $C$ with $0\in\relint(C)$. This nonstandard assumption is not restrictive when studying the volume $\lambda^n(S\oplus rC)$, since the Lebesgue measure is translation invariant.

    \begin{definition}
    Let $C\in\mathcal{C}^n$. The \textit{support function of $C$} is given by the formula $$h_C(y)\coloneq \sup_{x\in C}x\cdot y,\qquad y\in\mathbb{R}^n$$ 

     The \emph{radial function of $C$} is defined as $$\rho_{C}(x)\coloneq\max\{t\geq0 : tx\in C\}, \qquad x\in\mathbb{S}^{n-1}.$$
\end{definition}

\begin{rem} Let $C\in\mathcal{C}^n$.
    \begin{itemize}
        \item The function $x\mapsto h_C(x)$ is sublinear, hence convex, and Lipschitz with the property $h_{-C}(x)=h_{C}(-x)$ for all $x\in\mathbb{R}^n$. It is also true that $h_{aC}=ah_{C}$ for any $a>0$.
        \item We have $\rho_C(x)x\in C$ for any $x\in\mathbb{S}^{n-1}$, which in turn implies that $\rho_C(x)\leq\diam(C)$ for any $x\in\mathbb{S}^{n-1}.$
 \end{itemize}\end{rem}
\begin{definition}
    The mapping $d_H\colon \mathcal{C}^n\times\mathcal{C}^n\to[0, \infty)$ defined as $$d_H(C, K)\coloneq\max\big\{\sup_{x\in C}\dist(x, K), \sup_{y\in K}\dist(y, C)\big\},\qquad (C, K)\in\mathcal{C}^n\times\mathcal{C}^n,$$ is called the Hausdorff metric. The set $\mathcal{C}^n$ endowed with the mapping $d_H$ is a metric space.
\end{definition}
\begin{rem}
    For any pair $(C, K)\in\mathcal{C}^n\times\mathcal{C}^n$, it holds that $$d_H(C, K)=\Vert h_C-h_K\Vert_{\mathcal{C}(\mathbb{S}^{n-1})}=\inf\{t\geq 0 : C\subseteq K\oplus B(0, t),\, K\subseteq C\oplus B(0, t)\}.$$
\end{rem}

\medskip

Let us define the so-called mixed volumes (see \cite[Ch. 5]{rolf}).
\begin{definition}
    Let $C, K\subseteq\mathbb{R}^n$ be two compact convex sets. Then the measure $\lambda^n(K\oplus rC)$ can be expressed as follows:
\[
\lambda^n(K\oplus rC)
=\sum_{i=0}^n \binom{n}{i} V_n\bigl(K[i],\,C[n-i]\bigr)\,r^{\,n-i},\qquad r>0,
\]
where the coefficient $V_n\bigl(K[i],C[n-i]\bigr)$ is called the \emph{mixed volume of $i$ copies of $K$ and $n-i$ copies of $C$.}
\end{definition}

\begin{rem}
    If $C, K$ and $K^{\prime}$ are compact convex sets with $K\subseteq K^{\prime}$, then for any index $i\in\{0, \dots, n\}$, $x, y\in\mathbb{R}^n$, and for any $\alpha, \beta>0$, we have \begin{equation}\label{pp5}
    \begin{split}
        V_n\bigl(K[i],C[n-i]\bigr)&\leq V_n\bigl(K^{\prime}[i], C[n-i]\bigr)\\V_n\bigl(K[i],C[n-i]\bigr)&= V_n\bigl((x+K)[i], (y+C)[n-i]\bigr)\\ V_n\bigl((\alpha K)[i], (\beta C)[n-i]\bigr)&=\alpha^i \beta^{n-i}V_n\bigl(K[i], C[n-i]\bigr) \\ V_n\bigl(K[i], K[n-i]\bigr)&=\lambda^n(K).\end{split}\end{equation}
        
        Furthermore, the mapping $$(C, K)\mapsto V_n\bigl(K[i],C[n-i]\bigr)$$ is continuous from $\mathcal{C}^n\times\mathcal{C}^n$ to the reals.
\end{rem}

\section{Anisotropic $k$-dimensional Minkowski Content for Convex Bodies and Compact $\mathcal{C}^1$ $k$-graphs}\label{secDef}

Let us introduce \textit{$s$-dimensional anisotropic Minkowski content}.

\begin{definition}\label{def1} Let $C\in\mathcal{C}^n$, let  $E\subseteq\mathbb{R}^n$, and let $s\in [0, n]$. Putting $$\mathcal{M}^{s}_{r, C}(E)\coloneq \frac{\lambda^{n}(E\oplus rC)}{\omega_{n-s}r^{n-s}},\qquad r>0,$$ we define the \textit{$s$-dimensional lower} and \textit{upper $C$-anisotropic Minkowski content of $E$}, possibly taking the value $+\infty$, by \begin{align*}
    \mathcal{M}^s_{C}(E)_*\coloneq \liminf_{r\to0_+}\mathcal{M}^{s}_{r, C}(E)&& \text{and}&&\mathcal{M}^s_{C}(E)^*\coloneq \limsup_{r\to0_+}\mathcal{M}^{s}_{r, C}(E).
\end{align*} Whenever the two quantities are equal, that is, $\mathcal{M}^s_{C}(E)_*=\mathcal{M}^s_{C}(E)^*$, we define the \textit{$s$-dimensional $C$-anisotropic Minkowski content of $E$} by $$\mathcal{M}^s_{C}(E)\coloneq \mathcal{M}^s_{C}(E)^*.$$
\end{definition}

\begin{rem}\label{measurability}
As pointed out by the reviewer, if $C\in\mathcal{C}^n$ with $\dim(C)<n$ and $E\subseteq\mathbb{R}^n$ is measurable, then the Minkowski sum $E\oplus C$ need not be measurable. A simple example is the following: Let $F\subseteq\mathbb{R}$ be not $\lambda^1$-measurable, and let $$C\coloneqq \{0\}\times [-1, 1]\times [-1, 1], \qquad E\coloneqq F\times\{0\}\times\{0\}.$$ Then $E$ is $\lambda^3$-measurable, but $E\oplus C=F\times[-1, 1]\times [-1, 1]$ is not $\lambda^3$-measurable. Observe that in this case, the set $E$ is countably $\mathcal{H}^2$-rectifiable. Indeed, it is contained in a line, which implies that $E$ is $\mathcal{H}^2$-negligible. In Definition \ref{def1} we therefore work with the Lebesgue outer measure.

However, if $E$ is Borel, or more generally analytic, then $E\oplus C$ is measurable for every $C\in\mathcal{C}^n$ since the set $E\oplus C$ is analytic as the continuous image of the Borel (respectively analytic) set $E\times C$ under the continuous map $f\colon \mathbb{R}^n\times \mathbb{R}^n\to\mathbb{R}^n$ given by $f(x, y)\coloneqq x+y.$ 

On the other hand, if $\dim(C)=n$, then the Minkowski sum $E\oplus C$ is measurable even if $E$ itself is not measurable. Indeed, $$E\oplus \inter(C)=\overline{E}\oplus\inter(C).$$ The inclusion $\subseteq$ is trivial. Conversely, if $x=e+c\in \overline{E}\oplus\inter(C)$ for some $e\in\overline{E}$ and $c\in\inter(C)$, then there exists some $\varepsilon>0$ such that $B(c, \varepsilon)\subseteq \inter(C)$. Since $e\in\overline{E}$, there exists some $e^{\prime}\in E$ such that $|e-e^{\prime}|<\varepsilon$. But then $$x=e+c=e^{\prime}+(e-e^{\prime})+c\in E\oplus B(c, \varepsilon)\subseteq E\oplus \inter(C),$$ which completes the proof of the inclusion $\supseteq$. Furthermore, \begin{align}\label{pp20}
    \lambda^n(\overline{E}\oplus C)=\lambda^n\big(\overline{E}\oplus\inter(C)\big).
\end{align} (see \cite[p. 6, Eq. (10)]{Rataj5}). Next \begin{align*}
   \overline{E}\oplus\inter(C)=E\oplus \inter(C)\subseteq  E\oplus C\subseteq \overline{E}\oplus C.
\end{align*}
If $E$ is bounded, this together with \eqref{pp20} implies the measurability of $E\oplus C$. If $E$ is not bounded, we use that $$E\oplus C=\bigcup_{j\in\mathbb{N}}\Big(\big(E\cap B(0, j)\big)\oplus C\Big),$$ from which we conclude that the set $E\oplus C$ is again measurable.
\end{rem}

\begin{rem}\label{0inf}
    Let $C, C^{\prime}\in\mathcal{C}^n$ such that $\LO(C)=\LO(C^{\prime})$. Then $aC\subseteq C^{\prime}\subseteq bC$ for some $a, b>0$. Therefore, for a measurable set $E\subseteq\mathbb{R}^n$ and $s\in [0, n]$, we have $$a^{n-s}\mathcal{M}^s_{ar, C}(E)=\mathcal{M}^s_{r, aC}(E)\leq\mathcal{M}^s_{r, C^{\prime}}(E)\leq\mathcal{M}^s_{r, bC}(E)=b^{n-s}\mathcal{M}^s_{br, C}(E),\qquad r>0.$$ 

    Consequently, $$\mathcal{M}^s_{C}(E)_*=\infty \iff \mathcal{M}^s_{C^{\prime}}(E)_*=\infty$$ and $$\mathcal{M}^s_{C}(E)^*=0 \iff \mathcal{M}^s_{C^{\prime}}(E)^*=0.$$
\end{rem}

    The existence of the limit $\mathcal{M}^s_{C}(E)$ need not hold in general. For example, if $E$ is the \emph{Sierpi\'nski gasket} and $s=\log_2 3$, the limit fails to exist (see \cite[Example 3.3]{Rataj} for the isotropic case and \cite[Example 3.7]{frys2} for the anisotropic case).

    Even if $s\in [0, n]\cap \mathbb{N}$, the existence of the limit $\mathcal{M}^s_{C}(E)$ is not guaranteed. In \cite[3.2.40]{federer}, the author gives an example of a compact set $E\subseteq\mathbb{R}^2$ for which $$0=\mathcal{H}^1(E)=\mathcal{M}^1_{B(0,1)}(E)_*<\mathcal{M}^1_{B(0,1)}(E)^*=\infty.$$

   In our setting, we investigate the $k$-dimensional $C$-anisotropic content of a 
$k$-rectifiable, respectively countably $\mathcal{H}^k$-rectifiable, closed set 
for $k\in\{1, \dots, n-1\}$. 

The case $s=0$ and $n=\dim(C)$ is trivial, since for any measurable
set $E\subseteq\mathbb{R}^n$ and any $C\in\mathcal{C}^{n, n}$ (see \cite[Lemma~2.6]{frys2} for the case when $E$ is compact; the proof does not require compactness) we have
\begin{align}\label{0n}
    \mathcal{M}^0_{C}(E) = \frac{\mathcal{H}^0(E)\lambda^n(C)}{\omega_n}.
\end{align}

If $s=n$ and $E$ is compact, then for any $C\in\mathcal{C}^{n}$ it holds that \cite[Lemma~2.6]{frys2} $$\mathcal{M}^n_{C}(E)=\lim_{r\to0_+}\lambda^n(E\oplus rC)=\lambda^n(E).$$ More generally, if $C\in\mathcal{C}^{n}$ and $E\subseteq\mathbb{R}^n$, the mapping $r\mapsto \lambda^{n}(E\oplus rC)$ is nondecreasing. Hence, the limit $\mathcal{M}^n_C(E)$ always exists and coincides with $$\inf_{r>0}\lambda^{n}(E\oplus rC).$$

\begin{rem}\label{pp11}
    Let $C\in\mathcal{C}^n$, and let $E\subseteq\mathbb{R}^n$. Let $L\coloneqq\LO(C)$ and define $$\overline{E}^L\coloneqq \bigcap_{\sigma>0}(E\oplus \sigma C).$$ Note that the definition of $\overline{E}^L$ depends only on the linear span of $C$ and that if $\dim(C)=n,$ then $\overline{E}^{L}=\overline{E}$.
    
    Let $r>0$ and $s\in[0, n]$. Then $$E\oplus rC \subseteq \overline{E}^L\oplus rC\subseteq E\oplus r\sigma C\oplus rC, \qquad \sigma>0.$$ Consequently, $$\mathcal{M}^s_{r,C}(E)\leq \mathcal{M}^s_{r,C}(\overline{E}^L)\leq (1+\sigma)^{n-s}\mathcal{M}^s_{r(1+\sigma), C}(E),\qquad \sigma>0.$$

    Hence $$\mathcal{M}^s_{C}(E)_*\leq \mathcal{M}^s_{C}(\overline{E}^L)_*\leq (1+\sigma)^{n-s}\mathcal{M}^s_{C}(E)_*,\qquad \sigma>0$$ and $$\mathcal{M}^s_{C}(E)^*\leq \mathcal{M}^s_{C}(\overline{E}^L)^*\leq (1+\sigma)^{n-s}\mathcal{M}^s_{C}(E)^*,\qquad \sigma>0.$$
    Since $\sigma>0$ was arbitrary, we get \begin{align*}
        \mathcal{M}^s_{C}(E)_*=\mathcal{M}^s_{C}(\overline{E}^L)_*&&\text{and}&&\mathcal{M}^s_{C}(E)^*=\mathcal{M}^s_{C}(\overline{E}^L)^*.
    \end{align*} 
\end{rem}

\begin{definition}
    Let $C\in\mathcal{C}^n$, and let $S\subseteq\mathbb{R}^n$ be a countably $\mathcal{H}^k$-rectifiable set, where $k\in\{1, \dots, n-1\}.$ For $x\in S$ such that $\textup{T}^k_xS$ exists, put \begin{align*}
        \textup{N}^{n-k}_xS\coloneq(\textup{T}^k_xS)^{\perp},&&\textup{n}^{n-k}_{x}S\coloneq \mathbb{S}^{n-1}\cap\textup{N}^{n-k}_xS&&\text{and}&&C_x\coloneq P_{(\textup{T}^k_xS)^{\perp}}(C).
    \end{align*}
\end{definition}
 Clearly, $C_x$ is a convex body of dimension at most $n-k$ and by the area formula and Fubini's theorem \begin{equation}\label{Eq0}\begin{split}
     \mathcal{H}^{n-k}(C_x)&=\int\limits_{0}^{\infty}\int\limits_{\textup{n}^{n-k}_{x}S}\chi_{C_x}(s\nu)s^{n-k-1}\dH{n-k-1}(\nu)\dlambda{1}(s)\\&=\int\limits_{\textup{n}^{n-k}_{x}S}\int\limits_{0}^{\rho_{C_x}(\nu)}s^{n-k-1}\dlambda{1}(s)\dH{n-k-1}(\nu)\\&=\frac{1}{n-k}\int\limits_{\textup{n}^{n-k}_{x}S}\rho_{C_x}(\nu)^{n-k}\dH{n-k-1}(\nu).
 \end{split}
\end{equation}Furthermore, \begin{equation}\label{Eq17}
     \mathcal{H}^{n-k-1}(\textup{n}^{n-k}_{x}S)=(n-k)\omega_{n-k}.
 \end{equation}
 
\medskip 

In what follows we show that the limit $\mathcal{M}_C^k(S)$ always exists for sets $S$ which are compact subsets of $\mathcal{C}^1$ $k$-graphs and for convex bodies $C$ with $0\in\inter(C)$. In this case, the limit coincides with the following functional which depends in a natural way on $C$: \begin{equation}\label{limit}
\frac{1}{\omega_{n-k}}\int\limits_{S}\mathcal{H}^{n-k}(C_x)\dH{k}(x).
\end{equation} 
In contrast to Remark \ref{pp11}, the value of \eqref{limit} might change drastically when passing from $S$ to $\overline{S}^L$.

\medskip

The class of all sets which are compact subsets of $\mathcal{C}^1$ $k$-graphs contains, in particular, all convex bodies in $\mathbb{R}^n$ of dimension at most $k$.

For such convex bodies the situation is particularly simple.
If $K$ is a convex body of dimension at most $k$, then by \cite[Ch. 5]{rolf}
\begin{equation}\label{pp1}
    \lambda^n(K\oplus rC)
=\sum_{i=0}^k \binom{n}{i} V_n\bigl(K[i],\,C[n-i]\bigr)\,r^{\,n-i},\qquad r>0.
\end{equation}

Consequently,
\begin{equation}\label{pp3}
    \mathcal{M}^k_{C}(K)
=\binom{n}{k}\frac{V_n\bigl(K[k], C[n-k]\bigr)}{\omega_{n-k}}.
\end{equation}

Before proving that the $k$-dimensional $C$-anisotropic Minkowski content of a convex body of dimension at most $k$ coincides with the value of \eqref{limit}, we establish an equivalent form of \eqref{limit}.
\begin{lemma} Let $S$ be a countably $\mathcal{H}^k$-rectifiable set, where $k\in\{1, \dots, n-1\},$ and let $C\in\mathcal{C}^n$. Then \begin{equation*}\begin{split}\frac{1}{\omega_{n-k}}\int\limits_{S}\mathcal{H}^{n-k}(C_x)\dH{k}(x)=\int\limits_{S}\fint\limits_{\textup{n}^{n-k}_{x}S}\rho_{C_x}(\nu)^{n-k}\dH{n-k-1}(\nu)\dH{k}(x).
      \end{split}\end{equation*}
    \begin{proof}\newcommand\myeqf{\mathrel{\stackrel{\makebox[0pt]{\mbox{\normalfont\tiny (\ref{Eq17})}}}{=}}}\newcommand\myeqg{\mathrel{\stackrel{\makebox[0pt]{\mbox{\normalfont\tiny (\ref{Eq0})}}}{=}}} We have
    \begin{equation*}
    \begin{split}
        \int\limits_{S}\mathcal{H}^{n-k}&(C_x)\dH{k}(x)\\&\myeqg\frac{1}{n-k}\int\limits_{S}\int\limits_{\textup{n}^{n-k}_{x}S}\rho_{C_x}(\nu)^{n-k}\dH{n-k-1}(\nu)\dH{k}(x)\\&\myeqf\omega_{n-k}\int\limits_{S}\frac{1}{\mathcal{H}^{n-k-1}(\textup{n}^{n-k}_{x}S)}\int\limits_{\textup{n}^{n-k}_{x}S}\rho_{C_x}(\nu)^{n-k}\dH{n-k-1}(\nu)\dH{k}(x)\\&=\omega_{n-k}\int\limits_{S}\fint\limits_{\textup{n}^{n-k}_{x}S}\rho_{C_x}(\nu)^{n-k}\dH{n-k-1}(\nu)\dH{k}(x),
    \end{split}
\end{equation*}
       which completes the proof.
    \end{proof}
\end{lemma} 

We now prove that $\mathcal{M}^k_{C}(K)$ coincides with \eqref{limit} for all convex bodies $K$ with $\dim (K) \le k$.
It suffices to establish the statement for $K \in \mathcal{C}^{n,k}$, since otherwise both $\mathcal{M}^k_C(K)$ and the expression in \eqref{limit} vanish.

\begin{lemma}\label{inter}
    Let $K\subseteq\mathbb{R}^n$ be a convex body of dimension $k\in\{1, \dots, n-1\}$, and $C\in\mathcal{C}^n$. Then
        \begin{align}\label{pp2}
            \mathcal{M}^k_{C}(K)=\frac{1}{\omega_{n-k}}\int\limits_{K}\mathcal{H}^{n-k}(C_x)\dH{k}(x).
        \end{align}
        \begin{proof}
            By translation invariance of the Lebesgue measure, we may assume that $0\in\relint(K)$. Let $\varepsilon\in (0, 1)$, denote $L\coloneq\LO(K)$ and define \begin{align*}
                K_{\varepsilon}\coloneq\{x\in K : \dist(x, \partial_L K)>\varepsilon\}&&\text{and}&& K^{\prime}_{\varepsilon}\coloneq K\oplus B_{L}(0, \varepsilon).
            \end{align*}
            First, notice that by the Fubini theorem, we have
            $$\lambda^n(K\oplus rC)=\int\limits_{L}\mathcal{H}^{n-k}\big((K\oplus rC)\cap (x+L^{\perp})\big)\dH{k}(x), \qquad r>0.$$ 

            For $x\in L$, we have $$(K\oplus rC)\cap (x+L^{\perp})=\big\{x+rP_{L^{\perp}}(c) : c\in C,\, x-rP_L(c)\in K\big\}.$$ Moreover, if $x\in K_{\varepsilon}$ and $r>0$ is small enough (depending only on $\varepsilon$), then $x-rP_L(c)\in K,$ and hence $$(K\oplus rC)\cap (x+L^{\perp})=x+rP_{L^{\perp}}(C),$$ from which we conclude that for all sufficiently small $r>0$\begin{equation*}
                \begin{split}
                    \lambda^n(K\oplus rC)&\geq\int\limits_{K_{\varepsilon}}\mathcal{H}^{n-k}\big((K\oplus rC)\cap (x+L^{\perp})\big)\dH{k}(x)\\&=\int\limits_{K_{\varepsilon}}\mathcal{H}^{n-k}\big(P_{L^{\perp}}(rC)\big)\dH{k}\\&=r^{n-k}\int\limits_{K_{\varepsilon}}\mathcal{H}^{n-k}\big(P_{L^{\perp}}(C)\big)\dH{k},
                \end{split}
            \end{equation*}whence \begin{equation*}
                \begin{split}
                    \mathcal{M}^k_C(K)_*\geq\frac{1}{\omega_{n-k}}\mathcal{H}^k(K_\varepsilon)\mathcal{H}^{n-k}\big(P_{L^{\perp}}(C)\big), \qquad\varepsilon>0,
                \end{split}
            \end{equation*}from which it follows $$\mathcal{M}^k_C(K)_*\geq\frac{1}{\omega_{n-k}}\mathcal{H}^k(K)\mathcal{H}^{n-k}\big(P_{L^{\perp}}(C)\big)=\frac{1}{\omega_{n-k}}\int\limits_{K}\mathcal{H}^{n-k}(C_x)\dH{k}(x),$$ as $L=\textup{T}^k_xK$ for $\mathcal{H}^k$-a.a. $x\in K.$

            If $x\notin K_{\varepsilon}^{\prime}$ and $r>0$ is small enough, then $\mathcal{H}^{n-k}\big((K\oplus rC)\cap (x+L^{\perp})\big)=0$. So, for small enough $r>0$, it holds that \begin{equation*}
                \begin{split}
                    \lambda^n(K\oplus rC)&\leq\int\limits_{K_{\varepsilon}^{\prime}}\mathcal{H}^{n-k}\big((K\oplus rC)\cap (x+L^{\perp})\big)\dH{k}(x)\\&\leq\int\limits_{K_{\varepsilon}^{\prime}}\mathcal{H}^{n-k}\big(P_{L^{\perp}}(rC)\big)\dH{k}\\&=r^{n-k}\int\limits_{K^{\prime}_{\varepsilon}}\mathcal{H}^{n-k}\big(P_{L^{\perp}}(C)\big)\dH{k},
                \end{split}
            \end{equation*} where we used that $$(K\oplus rC)\cap (x+L^{\perp})\subseteq x+rP_{L^{\perp}}(C).$$ Hence \begin{equation*}
                \begin{split}
                    \mathcal{M}^k_C(K)^*\leq\frac{1}{\omega_{n-k}}\mathcal{H}^k(K^{\prime}_\varepsilon)\mathcal{H}^{n-k}\big(P_{L^{\perp}}(C)\big), \qquad\varepsilon>0,
                \end{split}
            \end{equation*}from which it follows $$\mathcal{M}^k_C(K)^*\leq\frac{1}{\omega_{n-k}}\mathcal{H}^k(K)\mathcal{H}^{n-k}\big(P_{L^{\perp}}(C)\big)=\frac{1}{\omega_{n-k}}\int\limits_{K}\mathcal{H}^{n-k}(C_x)\dH{k}(x),$$ which completes the proof.
        \end{proof}
\end{lemma}

Before proving an analogous result for $k$-rectifiable compact sets, we first examine the properties of the functional \eqref{limit}.

\begin{lemma}\label{continuity}
    The mapping $\Psi_{L}\colon \mathcal{C}^n\to\mathbb{R}$ given by $$\Psi_{L}\colon C\mapsto \mathcal{H}^{n-k}\big(P_{L^{\perp}}(C)\big)$$ is continuous uniformly with respect to $L\in G_k(\mathbb{R}^n)$. 
    
    In particular, the mapping $\Phi_S\colon \mathcal{C}^n\to\mathbb{R}$ given by $$\Phi_S\colon C\mapsto \frac{1}{\omega_{n-k}}\int\limits_{S}\mathcal{H}^{n-k}(C_x)\dH{k}(x)$$ is continuous for every countably $\mathcal{H}^k$-rectifiable set $S$ of finite $\mathcal{H}^k$-measure.
    \begin{proof}
        Let $\varepsilon\in(0, 1)$. If $K, C\in\mathcal{C}^n$ with $d_H(C, K)<\varepsilon,$ it holds that $C\subseteq K\oplus B(0, \varepsilon)$ and $K\subseteq C\oplus B(0, \varepsilon), $ from which we have \begin{equation*}
            \begin{split}
                \Psi_{L}(K)-\Psi_{L}(C)&\leq\Psi_{L}\big(C\oplus B(0, \varepsilon)\big)-\Psi_{L}(C)\\&=\mathcal{H}^{n-k}\Big(P_{L^{\perp}}(C)\oplus \varepsilon P_{L^{\perp}}\big(B(0, 1)\big)\Big)-\mathcal{H}^{n-k}\big(P_{L^{\perp}}(C)\big)\\&=\sum_{j=0}^{n-k-1}{n-k\choose j}\varepsilon^{n-k-j} V_{n-k}\big(P_{L^{\perp}}(C)[j], B_{L^{\perp}}(0, 1)[n-k-j]\big)\\&\leq\varepsilon\sum_{j=0}^{n-k-1}{n-k\choose j}\diam(C)^j V_{n-k}\big(B_{L^{\perp}}(0, 1)[j], B_{L^{\perp}}(0, 1)[n-k-j]\big)\\&=\varepsilon\sum_{j=0}^{n-k-1}{n-k\choose j}\diam(C)^j \omega_{n-k}.
            \end{split}
        \end{equation*}
        From symmetry, we obtain \begin{equation*}
            \begin{split}
                |\Psi_{L}(K)-\Psi_{L}(C)|\leq\varepsilon\sum_{j=0}^{n-k-1}{n-k\choose j}\big(\diam(C)+\diam(K)\big)^j \omega_{n-k},
            \end{split}
        \end{equation*}from which the first conclusion follows.

        Now, if $S$ is a countably $\mathcal{H}^k$-rectifiable set of finite $\mathcal{H}^k$-measure, then under the same assumptions  \begin{equation*}
            \begin{split}
                \left|\Phi_S(K)-\Phi_S(C)\right|&\leq\frac{1}{\omega_{n-k}}\int\limits_{S}|\Psi_{\textup{T}^k_xS}(K)-\Psi_{\textup{T}^k_xS}(C)|\dH{k}(x)\\&\leq \varepsilon\mathcal{H}^k(S)\sum_{j=0}^{n-k-1}{n-k\choose j}\big(\diam(C)+\diam(K)\big)^j,
            \end{split}
        \end{equation*} from which the second conclusion follows.
    \end{proof}
\end{lemma}
The following simple example illustrates why the assumption $\mathcal{H}^k(S) < \infty$ in Lemma \ref{continuity} is essential.

\begin{example}\label{E1}
    Let $x_n\coloneq\frac{1}{n}$ for $n\in\mathbb{N}$ and define \begin{align*}
        S_M\coloneq\bigcup_{n=1}^{M}\conv\left\{(-x_n, x_n), (x_n, x_n)\right\}&&\text{and}&&S\coloneq\{0\}\cup\bigcup_{M=1}^{\infty}S_M.
    \end{align*} Then $S$ is a countably $\mathcal{H}^1$-rectifiable compact subset of $\mathbb{R}^2$. 
    
    Let $C\coloneq[-1/2, 1/2]\times\{0\}$. Then clearly $$\mathcal{M}^1_{r, C}(S)=0\qquad\text{for every $r>0$.}$$ On the other hand, for $\mathcal{H}^1$-a.a. $x\in S$, it holds that $$\mathcal{H}^1(C_x)=0,$$ from which we see that $$\mathcal{M}^1_{C}(S)=\frac{1}{2}\int\limits_{S}\mathcal{H}^{1}(C_x)\dH{1}(x)=0.$$ However, if $\{C^{\varepsilon}\}_{\varepsilon>0}$ is a family of full-dimensional convex bodies in $\mathbb{R}^2$ satisfying $C^{\varepsilon}\xrightarrow[\varepsilon\to0_+]{d_H} C$, we have 
$$\mathcal{M}^{1}_{C^{\varepsilon}}(S)_*\geq\sup_{M\in\mathbb{N}}\mathcal{M}^{1}_{C^{\varepsilon}}(S_M)=\mathcal{H}^1\big(P_{\LO(e_2)}(C^{\varepsilon})\big)\sum_{n=1}^{\infty}\frac{1}{n}=\infty.$$ Here, $e_2=(0, 1)$. Similarly, $$\int\limits_S\mathcal{H}^1(C^{\varepsilon}_x)\dH{1}(x)\geq\sup_{M\in\mathbb{N}}\int\limits_{S_M}\mathcal{H}^1(C^{\varepsilon}_x)\dH{1}(x)=\mathcal{H}^1\big(P_{\LO(e_2)}(C^{\varepsilon})\big)\sum_{n=1}^{\infty}\frac{2}{n}=\infty,$$ from which we conclude that $$\frac{1}{2}\int\limits_{S}\mathcal{H}^{1}(C^{\varepsilon}_x)\dH{1}(x)\not\rightarrow\frac{1}{2}\int\limits_{S}\mathcal{H}^{1}(C_x)\dH{1}(x)$$ as $\varepsilon\to0_+.$
\end{example}

\begin{rem}\label{0inf2}
    Let $S\subseteq\mathbb{R}^n$ be a countably $\mathcal{H}^k$-rectifiable set, where $k\in\{1, \dots, n-1\}$. The functional $$\Phi_S\colon \mathcal{C}^n\to [0, \infty], \qquad \Phi_S(C)=\frac{1}{\omega_{n-k}}\int\limits_{S}\mathcal{H}^{n-k}(C_x)\dH{k}(x)$$ enjoys the following property: For $C, C^{\prime}\in\mathcal{C}^n$ with $\LO(C)=\LO(C^{\prime})$, it holds that $$a^{n-k}\Phi_S(C)=\Phi_S(aC)\leq\Phi_S(C^{\prime})\leq\Phi_S(bC)=b^{n-k}\Phi_S(C),$$ where $a$ and $b$ are positive reals such that $aC\subseteq C^{\prime}\subseteq bC.$ In particular, $$\Phi_S(C)=0\iff \Phi_S(C^{\prime})=0$$ and $$\Phi_S(C)=\infty\iff \Phi_S(C^{\prime})=\infty.$$ Furthermore, if $\dim(C)=n$, then $$\Phi_S(C)=\infty \iff \mathcal{H}^k(S)=\infty$$ and $$\Phi_S(C)=0 \iff \mathcal{H}^k(S)=0$$ since $\Phi_S\big(B(0, 1)\big)=\mathcal{H}^k(S).$
\end{rem}

\begin{lemma}\label{gong}
    Let $G\subseteq\mathbb{R}^k$ be open, $K\subseteq G$ be a nonempty compact set, and let $f\in\mathcal{C}^1(G; \mathbb{R}^{n-k})$. Define $$g(x)\coloneq\big(x, f(x)\big), \qquad x\in G.$$ Then the mapping
\[
g(x)\mapsto \textup{N}^{n-k}_{g(x)}g(G)
\]
is continuous from $g(K)$ to $G_{n-k}(\mathbb R^n)$.
    
    In particular, the mapping $$w\mapsto \mathcal{H}^{n-k}\big(P_{\textup{N}^{n-k}_w g(G)}(C)\big)$$ is continuous from $g(K)$ to $\mathbb{R}$, where $C\in\mathcal{C}^{n, n}$. \begin{proof}
        Since $\rank\big(Dg(x)\big)=k$ for any $x\in G$, we have $$P_{\textup{T}^k_{g(x)}g(G)}=Dg(x)\big(Dg(x)^{\top}Dg(x)\big)^{-1}Dg(x)^{\top}, \qquad x\in G.$$ Consequently, for any $x, y\in K$\begin{equation*}
        \begin{split}
           P_{\textup{T}^k_{g(x)}g(G)}-P_{\textup{T}^k_{g(y)}g(G)}&=\big(Dg(x)-Dg(y)\big)\big(Dg(x)^{\top}Dg(x)\big)^{-1}Dg(x)^{\top}\\&\quad\quad+Dg(y)\Big(\big(Dg(x)^{\top}Dg(x)\big)^{-1}-\big(Dg(y)^{\top}Dg(y)\big)^{-1}\Big)Dg(x)^{\top} \\&\quad\quad+Dg(y)\big(Dg(y)^{\top}Dg(y)\big)^{-1}\big(Dg(x)-Dg(y)\big)^{\top}.
        \end{split}
        \end{equation*}
        We see that $$Dg(x)^{\top}Dg(x)=I_k+Df(x)^{\top}Df(x),$$ from which we conclude that all eigenvalues of $Dg(x)^{\top}Dg(x)$ are bounded from below by $1$. Thus the eigenvalues of $\big(Dg(x)^{\top}Dg(x)\big)^{-1}$ are bounded from above by $1$, which in turn implies $$\left\Vert\big(Dg(x)^{\top}Dg(x)\big)^{-1}\right\Vert\leq 1.$$ Using this and $M\coloneq\sup_{x\in K}\Vert Dg(x)\Vert<\infty$, we conclude that \begin{equation*}
        \begin{split}\Vert P_{\textup{N}^{n-k}_{g(x)}g(G)}-P_{\textup{N}^{n-k}_{g(y)}g(G)}\Vert&=\Vert P_{\textup{T}^k_{g(x)}g(G)}-P_{\textup{T}^k_{g(y)}g(G)}\Vert\\&\leq 2M\Vert Dg(x)-Dg(y)\Vert\\&\quad\quad+M^ {2}\left\Vert \big(Dg(x)^{\top}Dg(x)\big)^{-1}-\big(Dg(y)^{\top}Dg(y)\big)^{-1}\right\Vert\\&\leq 2M\Vert Dg(x)-Dg(y)\Vert\\&\quad\quad+M^ {2}\left\Vert Dg(y)^{\top}Dg(y)-Dg(x)^{\top}Dg(x)\right\Vert\\&\leq (2M+2M^{3})\Vert Dg(x)-Dg(y)\Vert,
        \end{split}
        \end{equation*} from which the first conclusion follows.

        Now, let $C\in\mathcal{C}^{n, n}$. It suffices to show that the mapping $$L\mapsto \mathcal{H}^{n-k}\big(P_L(C)\big)$$ is continuous from $G_{n-k}(\mathbb{R}^n)$ to $\mathbb{R}$. To this end, fix $L, W\in G_{n-k}(\mathbb{R}^n)$. For an arbitrary pair $(c_1, c_2)\in C\times C$\begin{equation*}
            \begin{split}|P_W(c_1)-P_L(c_2)|&\leq  |P_W(c_1)-P_W(c_2)|+|P_W(c_2)-P_L(c_2)|\\&\leq |c_1-c_2|+\diam(C)\Vert P_W-P_L\Vert,\end{split}
        \end{equation*} which yields $$d_H\big(P_W(C), P_L(C)\big)\leq\diam(C)\Vert P_W-P_L\Vert.$$ Using this, we see that if $L_i\xrightarrow{i\to\infty} L$ in $G_{n-k}(\mathbb{R}^n),$ then $P_{L_{i}}(C)\xrightarrow{i\to\infty}P_{L}(C)$ in $\mathcal{C}^n$, which in turn implies $\mathcal{H}^{n-k}\big(P_{L_{i}}(C)\big)\xrightarrow{i\to\infty}\mathcal{H}^{n-k}\big(P_{L}(C)\big).$
    \end{proof}
\end{lemma}

\begin{lemma}\label{densit}
Let $K\subseteq\mathbb{R}^k$ be compact and let $F\in L^1(K)$ be nonnegative.
For every $\varepsilon>0$ and every $\theta\in(0, 1)$ there exist a compact set
$K_{\varepsilon, \theta}\subseteq K$ and a number
$\rho_{\varepsilon,\theta}>0$ such that
\[
\int\limits_{K\setminus K_{\varepsilon,\theta}}F(y)\dlambda{k}(y)<\varepsilon
\]
and
\[
\lambda^k\big(Q_{\mathbb{R}^k}(y, t)\setminus K\big)\leq \theta\,\lambda^k\big(Q_{\mathbb{R}^k}(y, t)\big)
\]
for every $y\in K_{\varepsilon, \theta}$ and every
$0<t<\rho_{\varepsilon, \theta}$.

Moreover, 
\[
Q_{\mathbb{R}^k}(y, t/2)\subseteq K\oplus B_{\mathbb R^k}(0, 3\sqrt{k}\theta^{1/k}t)
\]
for every $y\in K_{\varepsilon, \theta}$ and every
$0<t<\rho_{\varepsilon, \theta}$.
\end{lemma}
\begin{proof}
Since $K$ is compact, it has finite Lebesgue measure. By the Lebesgue density
theorem, for $\lambda^k$-a.e. $y\in K$ we have
\[
\frac{\lambda^k\big(Q_{\mathbb{R}^k}(y, t)\setminus K\big)}{\lambda^k\big(Q_{\mathbb{R}^k}(y, t)\big)}\xrightarrow{t\to0_+}0.
\]
By Egorov's theorem and the absolute continuity of the integral of $F$, there
exists a compact set $K_{\varepsilon, \theta}\subseteq K$ such that
\[
\int\limits_{K\setminus K_{\varepsilon, \theta}}F(y)\dlambda{k}(y)<\varepsilon
\]
and such that the above convergence is uniform on $K_{\varepsilon, \theta}$.
Thus there exists $\rho_{\varepsilon, \theta}>0$ such that
\[
\lambda^k\big(Q_{\mathbb{R}^k}(y, t)\setminus K\big)\leq \theta\,\lambda^k\big(Q_{\mathbb{R}^k}(y, t)\big)
\]
for every $y\in K_{\varepsilon, \theta}$ and every
$0<t<\rho_{\varepsilon, \theta}$.

It remains to prove the inclusion. Let $y\in K_{\varepsilon, \theta}$,
let $0<t<\rho_{\varepsilon, \theta}$, and let $z\in Q_{\mathbb{R}^k}(y, t/2)$. Since
$y\in K$, we have
\[
\dist(z, K)\leq |z-y|\leq \frac{\sqrt{k}}{2}t.
\]
Thus, the desired estimate is trivial if
$\theta^{1/k}\geq 1/6$. We may therefore assume that
$$\theta^{1/k}<1/6.$$

Suppose, to the contrary, that $z\notin K\oplus B_{\mathbb{R}^k}(0, 3\sqrt{k}\theta^{1/k}t).$ Equivalently,
\[
\dist(z, K)>3\sqrt{k}\theta^{1/k}t.
\]
Then the cube
\[
Q_{\mathbb{R}^k}\big(z, 3\theta^{1/k}t/2\big)
\]
is contained in $Q_{\mathbb{R}^k}(y, t)\setminus K$. Indeed, since $z\in Q_{\mathbb{R}^k}(y,t/2)$ and $\theta^{1/k}<1/6$, we have \[ \frac{3}{2}\theta^{1/k}t<\frac{t}{4}, \] and hence \[ Q_{\mathbb{R}^k}\big(z,3\theta^{1/k}t/2\big)\subseteq Q_{\mathbb{R}^k}(y,t). \] Moreover, if $u\in Q_{\mathbb{R}^k}\big(z,3\theta^{1/k}t/2\big)$, then \[ |u-z|\leq \sqrt{k}\,\frac{3}{2}\theta^{1/k}t <3\sqrt{k}\theta^{1/k}t<\dist(z,K), \] and therefore $u\notin K$.

Hence
\[
\lambda^k\big(Q_{\mathbb{R}^k}(y, t)\setminus K\big)
\geq
(3\theta^{1/k}t)^k=\theta (3t)^k.
\]

On the other hand,
\[
\lambda^k\big(Q_{\mathbb{R}^k}(y, t)\setminus K\big)
\leq
\theta\,\lambda^k\big(Q_{\mathbb{R}^k}(y, t)\big)
=
\theta(2t)^k,
\]
which gives a contradiction. Therefore
\[
\dist(z, K)\leq 3\sqrt{k}\theta^{1/k}t
\]
for every $z\in Q_{\mathbb{R}^k}(y, t/2)$, which proves
\[
Q_{\mathbb{R}^k}(y, t/2)\subseteq K\oplus B_{\mathbb R^k}(0, 3\sqrt{k}\theta^{1/k}t).
\]
\end{proof}

The following two lemmata generalize the results of Federer (\cite{federer}) and Lussardi and Villa (\cite{Villa2}). 

We note that the assumption $C\in\mathcal{C}^{n,n}$ in the following lemma can be relaxed to $C\in\mathcal{C}^{n}$, as follows from Proposition~\ref{LB} proved below.

\begin{lemma}\label{krecti}
Let $C\in\mathcal{C}^{n, n}$, and let $S\subseteq\mathbb{R}^n$ be a compact
subset of a $\mathcal{C}^1$ $k$-graph, where $k\in\{1, \dots, n-1\}$.
Then
\[
\mathcal{M}^k_{C}(S)_*
\geq
\frac{1}{\omega_{n-k}}
\int\limits_S\mathcal{H}^{n-k}(C_x)\dH{k}(x).
\]
\begin{proof}
Choose $\alpha, \beta>0$ such that
\[
B(0, \alpha)\subseteq C\subseteq B(0, \beta).
\]

After a rotation of $\mathbb{R}^n$, we may assume that
\[
S\subseteq \big\{\big(y, h(y)\big) : y\in G\big\},
\]
where $G\subseteq\mathbb R^k$ is an open bounded set and
$h\in\mathcal C^1(G; \mathbb R^{n-k})$. For brevity, define
\[
g(y)\coloneq \big(y, h(y)\big),\qquad y\in G,
\]
and put
\[
K\coloneq g^{-1}(S).
\]
Then $K$ is compact and $K\subseteq G$.

The function
\[
y\mapsto J_kg(y)\mathcal H^{n-k}
\big(P_{\textup{N}^{n-k}_{g(y)}g(G)}(C)\big)
\]
is continuous on $G$.

It is enough to prove
\[
\mathcal M_C^k(S)_*
\geq
\frac1{\omega_{n-k}}
\int\limits_K
J_kg(y)\mathcal H^{n-k}
\big(P_{\textup{N}^{n-k}_{g(y)}g(G)}(C)\big)\dlambda{k}(y).
\]
Indeed, by the area formula and by the fact that, for $\mathcal{H}^k$-a.e.
point of $S$, the approximate tangent space of $S$ agrees with the tangent
space of the ambient $\mathcal{C}^1$ $k$-graph,
\[
\int\limits_K
J_kg(y)\mathcal H^{n-k}
\big(P_{\textup{N}^{n-k}_{g(y)}g(G)}(C)\big)\dlambda{k}(y)
=
\int\limits_S \mathcal H^{n-k}(C_x)\dH{k}(x).
\]

Fix $\sigma\in(0, 1)$ and choose
\[
\eta\in\left(0, \min\left\{1, (4\beta)^{-1}\right\}\right).
\]
Further, choose $\delta_0>0$ such that
\[
K\oplus B_{\mathbb R^k}(0, \delta_0)\subseteq G,
\]and choose $\theta\in(0, 1)$ so small that
\[
\textup{Lip}(g)\theta^{1/k}\eta^{-1}6\sqrt{k}
\leq
\frac{\sigma\alpha}{4},
\]
where
$\textup{Lip}(g)$ is a Lipschitz constant of $g$ on $K\oplus B_{\mathbb R^k}(0, \delta_0).$

Apply Lemma~\ref{densit} to the compact set $K$ and to the function
\[
y\mapsto J_kg(y)\mathcal H^{n-k}
\big(P_{\textup{N}^{n-k}_{g(y)}g(G)}(C)\big).
\]
Thus, there exists a compact set $K_{\sigma, \theta}\subseteq K$ such that
\[
\int\limits_{K\setminus K_{\sigma,\theta}}
J_kg(y)\mathcal{H}^{n-k}
\big(P_{\textup{N}^{n-k}_{g(y)}g(G)}(C)\big)\dlambda{k}(y)
<\sigma
\]
and there exists $\rho>0$ such that
\[
\lambda^k\big(Q_{\mathbb{R}^k}(y, t)\setminus K\big)
\leq
\theta\lambda^k\big(Q_{\mathbb{R}^k}(y, t)\big)
\]
for every $y\in K_{\sigma,\theta}$ and every $0<t<\rho$. Moreover,
\[
Q_{\mathbb{R}^k}(y, t/2)
\subseteq
K\oplus B_{\mathbb R^k}\big(0, 3\sqrt{k}\theta^{1/k}t\big)
\]
for every such $y$ and $t$.

Let $r>0$ be small and set
\[
\ell\coloneq \eta^{-1}r.
\]
We assume that $r$ is so small that
\[
2\ell<\rho,
\qquad
\sqrt{k}\ell<\delta_0.
\]

Let $\mathcal{Q}_r$ be the finite family of cubes of a fixed grid in
$\mathbb R^k$ of side length $\ell$ which intersect $K_{\sigma,\theta}$.
For each $Q\in\mathcal{Q}_r$, choose a point
\[
y_Q\in Q\cap K_{\sigma,\theta}.
\]
Define the affine map
\[
g_Q(z)\coloneq g(y_Q)+Dg(y_Q)(z-y_Q),
\qquad z\in\mathbb R^k.
\]
Since $g$ is a graph map, the first $k$ coordinates of $g_Q(z)$ are exactly
the coordinates of $z$.

For each cube $Q\in\mathcal Q_r$, define $Q^-$ to be the cube with the same
center as $Q$ and side length $\ell-4\beta r>0$; this is positive because
$\ell=\eta^{-1}r$ and $\eta<(4\beta)^{-1}$.

We first prove that
\[
g_Q(Q^-)\oplus rC
\subseteq
S\oplus (1+\sigma)rC
\qquad\text{for all }Q\in\mathcal{Q}_r
\]
provided $r>0$ is sufficiently small.

Let $Q\in\mathcal{Q}_r$ and let $z\in Q^-$. Since $Q$ has side length $\ell$
and $y_Q\in Q$, we have
\[
Q\subseteq Q_{\mathbb{R}^k}(y_Q, \ell).
\]
Hence \(z\in Q_{\mathbb{R}^k}(y_Q,\ell)=Q_{\mathbb{R}^k}(y_Q,2\ell/2)\). Applying the last part of
Lemma~\ref{densit} with \(t=2\ell<\rho\), we find \(z_0\in K\) such that
\[
|z-z_0|\leq 6\sqrt{k}\theta^{1/k}\ell.
\]
Thus
\[
|g(z)-g(z_0)|
\leq
6\sqrt{k}\textup{Lip}(g)\theta^{1/k}\ell
=
6\sqrt{k}\textup{Lip}(g)\theta^{1/k}\eta^{-1}r
\leq
\frac{\sigma\alpha}{4}r.
\]
On the other hand, by the uniform continuity of \(Dg\) on
\(K\oplus B_{\mathbb R^k}(0,\delta_0)\), for all sufficiently small \(r\),
\[
|g(z)-g_Q(z)|
\leq
\frac{\sigma\alpha}{4}r.
\]
Consequently,
\[
|g_Q(z)-g(z_0)|
\leq
\frac{\sigma\alpha}{2}r
<
\sigma\alpha r.
\]
Since \(B(0,\alpha)\subseteq C\), we obtain
\[
g_Q(z)\in g(z_0)+\sigma rC\subseteq S\oplus \sigma rC.
\]
Therefore
\[
g_Q(z)+rC
\subseteq
S\oplus \sigma rC\oplus rC
\subseteq
S\oplus(1+\sigma)rC.
\]
This proves the inclusion.

Next, we prove that the sets
\[
g_Q(Q^-)\oplus rC,
\qquad Q\in\mathcal Q_r,
\]
are pairwise disjoint. Let
\[
\pi:\mathbb R^n\to\mathbb R^k
\]
denote the projection onto the first \(k\) coordinates. Since \(C\subseteq B(0, \beta)\),
if
\[
w\in g_Q(Q^-)\oplus rC,
\]
then
\[
\pi(w)\in Q^-\oplus B_{\mathbb R^k}(0, \beta r).
\]
By the definition of \(Q^-\), the set $Q^-\oplus B_{\mathbb R^k}(0, \beta r)$ is contained in the interior of \(Q\).
The interiors of distinct cubes of the grid are disjoint. Hence the sets
\(g_Q(Q^-)\oplus rC\), \(Q\in\mathcal Q_r\), are pairwise disjoint.

It follows that
\[
\lambda^n\big(S\oplus(1+\sigma)rC\big)
\geq
\sum_{Q\in\mathcal Q_r}
\lambda^n\big(g_Q(Q^-)\oplus rC\big).
\]

For each \(Q\in\mathcal Q_r\), the set \(g_Q(Q^-)\) is a compact convex set
of dimension \(k\). Using the polynomial expansion \eqref{pp1} for \(\lambda^n(g_Q(Q^-)\oplus rC)\), together with \eqref{pp3} and
the identity established in Lemma~\ref{inter}, see \eqref{pp2}, we obtain
\[
\frac{\lambda^n(g_Q(Q^-)\oplus rC)}{r^{n-k}}
\geq
\int\limits_{g_Q(Q^-)}
\mathcal{H}^{n-k}
\big(P_{\textup{N}^{n-k}_{g(y_Q)}g(G)}(C)\big)\dH{k}(y).
\]
By the area formula for the injective affine map \(g_Q\),
\[
\int\limits_{g_Q(Q^-)}
\mathcal H^{n-k}
\big(P_{\textup{N}^{n-k}_{g(y_Q)}g(G)}(C)\big)\dH{k}(y)
=
\int\limits_{Q^-}
J_kg_Q(z)
\mathcal H^{n-k}
\big(P_{\textup{N}^{n-k}_{g(y_Q)}g(G)}(C)\big)\dlambda{k}(z).
\]
Since \(Dg_Q=Dg(y_Q)\), we have
\[
J_kg_Q(z)=J_kg(y_Q)
\]
for every \(z\in Q^-\). Hence
\[
\begin{split}
\int\limits_{Q^-}
J_kg_Q(z)
\mathcal H^{n-k}
\big(P_{\textup{N}^{n-k}_{g(y_Q)}g(G)}(C)\big)\dlambda{k}(z)
=
J_kg(y_Q)
\mathcal H^{n-k}
\big(P_{\textup{N}^{n-k}_{g(y_Q)}g(G)}(C)\big)\lambda^k(Q^-).
\end{split}
\]
Thus
\begin{equation}\label{pp7}
\begin{split}
\frac{\lambda^n\big(S\oplus(1+\sigma)rC\big)}{r^{n-k}}
&\geq
\sum_{Q\in\mathcal{Q}_r}
J_kg(y_Q)
\mathcal{H}^{n-k}
\big(P_{\textup{N}^{n-k}_{g(y_Q)}g(G)}(C)\big)\lambda^k(Q^-).
\end{split}
\end{equation}

Set
\[
U_r\coloneq\bigcup_{Q\in\mathcal Q_r}Q^-.
\]
By the uniform continuity of the function
\[
y\mapsto
J_kg(y)\mathcal H^{n-k}
\big(P_{\textup{N}^{n-k}_{g(y)}g(G)}(C)\big)
\]
on \(K\oplus B_{\mathbb R^k}(0, \delta_0)\), for all sufficiently small \(r>0\),
\begin{equation*}
\begin{split}
\Big|
J_kg(z)\mathcal{H}^{n-k}
\big(P_{\textup{N}^{n-k}_{g(z)}g(G)}(C)\big)-
J_kg(y_Q)\mathcal{H}^{n-k}
\big(P_{\textup{N}^{n-k}_{g(y_Q)}g(G)}(C)\big)
\Big|
\leq\sigma
\end{split}
\end{equation*}
whenever \(z\in Q^-\). Therefore
\begin{equation}\label{pp8}
\begin{split}
&\sum_{Q\in\mathcal Q_r}
J_kg(y_Q)\mathcal{H}^{n-k}
\big(P_{\textup{N}^{n-k}_{g(y_Q)}g(G)}(C)\big)
\lambda^k(Q^-)
\\
&\quad\geq
\int\limits_{U_r}
J_kg(z)\mathcal{H}^{n-k}
\big(P_{\textup{N}^{n-k}_{g(z)}g(G)}(C)\big)\dlambda{k}(z)
-\sigma\lambda^k(U_r).
\end{split}
\end{equation}
For all small enough $r>0$, we have \(U_r\subseteq K\oplus B_{\mathbb{R}^k}(0, \delta_0)\), and hence
\begin{equation}\label{pp10}
    \lambda^k(U_r)\leq\lambda^k\big(K\oplus B_{\mathbb{R}^k}(0, \delta_0)\big).
\end{equation}

Further, let $R$ be a fixed cube of side length $L$ containing $K\oplus B_{\mathbb{R}^k}(0, \delta_0)$ and let $\mathcal S_\ell$ denote the skeleton of the whole grid of side length $\ell$. 
The cube $R$ and its side length $L$ are fixed independently of $r$.

We have
\[
K_{\sigma,\theta}\setminus U_r
\subseteq
\big\{w\in R:\dist(w,\mathcal S_\ell)<2\sqrt{k}\beta r\big\}.
\]
Indeed, if $w\in K_{\sigma,\theta}\setminus U_r$, then $w$ belongs to some
grid cube $Q\in\mathcal Q_r$, but $w\notin Q^-$. Hence its supremum distance
from $\partial Q$ is less than $2\beta r$. Therefore its Euclidean distance
from the grid skeleton is less than $2\sqrt{k}\beta r$.

In each coordinate
direction, at most
\[
\frac{L}{\ell}+2\leq 3\frac{L}{\ell}
\]
grid hyperplanes have their \(2\sqrt{k}\beta r\)-neighbourhood intersecting
\(R\), provided \(r>0\) is sufficiently small. The \(2\sqrt{k}\beta r\)-
neighbourhood of one such hyperplane intersects \(R\) in a strip of volume at
most
\[
4\sqrt{k}\beta r\,L^{k-1}.
\]
Therefore, in one coordinate direction, the total volume of the corresponding
strips is at most
\[
3\frac{L}{\ell}\,4\sqrt{k}\beta r\,L^{k-1}
=
12\sqrt{k}\beta L^k\frac{r}{\ell}
=
12\sqrt{k}\beta L^k\eta.
\]
Since there are \(k\) coordinate directions, we obtain
\[
\lambda^k(K_{\sigma,\theta}\setminus U_r)
\leq
12k\sqrt{k}\beta L^k\eta.
\]

Set
\[
M\coloneq
\sup_{z\in K\oplus B_{\mathbb{R}^k}(0, \delta_0)}
J_kg(z)\mathcal{H}^{n-k}
\big(P_{\textup{N}^{n-k}_{g(z)}g(G)}(C)\big)
<\infty.
\]
Then
\begin{equation}\label{pp9}
\begin{split}
&\int\limits_{U_r}
J_kg(z)\mathcal H^{n-k}
\big(P_{\textup{N}^{n-k}_{g(z)}g(G)}(C)\big)\dlambda{k}(z)
\\
&\quad\geq
\int\limits_{K_{\sigma,\theta}}
J_kg(z)\mathcal H^{n-k}
\big(P_{\textup{N}^{n-k}_{g(z)}g(G)}(C)\big)\dlambda{k}(z)
-
M 12k\sqrt{k}\beta L^k\eta.
\end{split}
\end{equation}
Combining \eqref{pp7}, \eqref{pp8}, \eqref{pp10} and \eqref{pp9}, we get
\[
\begin{split}
\frac{\lambda^n\big(S\oplus(1+\sigma)rC\big)}{r^{n-k}}
&\geq
\int\limits_{K_{\sigma,\theta}}
J_kg(z)\mathcal{H}^{n-k}
\big(P_{\textup{N}^{n-k}_{g(z)}g(G)}(C)\big)\dlambda{k}(z)
\\
&\quad-
M 12k\sqrt{k}\beta L^k\eta
-
\sigma\lambda^k\big(K\oplus B_{\mathbb{R}^k}(0, \delta_0)\big)
\end{split}
\]
for all sufficiently small \(r>0\).

Since
\[
(1+\sigma)^{n-k}\mathcal M^k_{(1+\sigma)r,C}(S)
=
\frac{\lambda^n(S\oplus(1+\sigma)rC)}
{\omega_{n-k}r^{n-k}},
\]
the preceding estimate can be written as
\[
\begin{split}
\omega_{n-k}(1+\sigma)^{n-k}
\mathcal M^k_{(1+\sigma)r,C}(S)
&\geq
\int\limits_{K_{\sigma,\theta}}
J_kg(z)\mathcal H^{n-k}
\big(P_{\textup{N}^{n-k}_{g(z)}g(G)}(C)\big)\dlambda{k}(z)
\\
&\quad-
M 12k\sqrt{k}\beta L^k\eta
-
\sigma\lambda^k\big(K\oplus B_{\mathbb{R}^k}(0, \delta_0)\big).
\end{split}
\]
Taking the limit inferior as \(r\to0_+\) yields
\[
\begin{split}
\omega_{n-k}(1+\sigma)^{n-k}
\mathcal M^k_C(S)_*
&\geq
\int\limits_{K_{\sigma,\theta}}
J_kg(z)\mathcal H^{n-k}
\big(P_{\textup{N}^{n-k}_{g(z)}g(G)}(C)\big)\dlambda{k}(z)
\\
&\quad-
M 12k\sqrt{k}\beta L^k\eta
-
\sigma\lambda^k\big(K\oplus B_{\mathbb{R}^k}(0, \delta_0)\big).
\end{split}
\]
By the choice of \(K_{\sigma,\theta}\),
\[
\int\limits_{K\setminus K_{\sigma,\theta}}
J_kg(z)\mathcal H^{n-k}
\big(P_{\textup{N}^{n-k}_{g(z)}g(G)}(C)\big)\dlambda{k}(z)
<\sigma.
\]
Therefore
\[
\begin{split}
\omega_{n-k}(1+\sigma)^{n-k}
\mathcal M^k_C(S)_*
&\geq
\int\limits_K
J_kg(z)\mathcal H^{n-k}
\big(P_{\textup{N}^{n-k}_{g(z)}g(G)}(C)\big)\dlambda{k}(z)
\\
&\quad-
\sigma
-
M 12k\sqrt{k}\beta L^k\eta
-
\sigma\lambda^k\big(K\oplus B_{\mathbb{R}^k}(0, \delta_0)\big).
\end{split}
\]
Letting \(\eta\to0_+\) and then \(\sigma\to0_+\), we conclude that
\[
\omega_{n-k}\mathcal{M}^k_C(S)_*
\geq
\int\limits_K
J_kg(z)\mathcal{H}^{n-k}
\big(P_{\textup{N}^{n-k}_{g(z)}g(G)}(C)\big)\dlambda{k}(z).
\]
Thus
\[
\mathcal{M}^k_C(S)_*
\geq
\frac1{\omega_{n-k}}
\int\limits_K
J_kg(z)\mathcal H^{n-k}
\big(P_{\textup{N}^{n-k}_{g(z)}g(G)}(C)\big)\dlambda{k}(z).
\]
Finally, the area formula gives
\[
\int\limits_K
J_kg(z)\mathcal{H}^{n-k}
\big(P_{\textup{N}^{n-k}_{g(z)}g(G)}(C)\big)\dlambda{k}(z)
=
\int\limits_S\mathcal{H}^{n-k}(C_x)\dH{k}(x).
\]
Therefore
\[
\mathcal{M}^k_C(S)_*
\geq
\frac1{\omega_{n-k}}
\int\limits_S\mathcal{H}^{n-k}(C_x)\dH{k}(x),
\]
as required.
\end{proof}
\end{lemma}

The following result provides an upper bound for the upper 
$k$-dimensional anisotropic Minkowski content of a compact $\mathcal{C}^1$ 
$k$-graph. Our argument is inspired by \cite[Lemma 3.2.38]{federer}.
\begin{lemma}\label{krect}
        Let $C\in\mathcal{C}^n$, and let $S\subseteq\mathbb{R}^n$ be a compact subset of a $\mathcal{C}^1$ $k$-graph, where $k\in\{1, \dots, n-1\}$. Then we have \begin{align*}
            \mathcal{M}^k_{C}(S)^*\leq\frac{1}{\omega_{n-k}}\int\limits_{S}\mathcal{H}^{n-k}(C_x)\dH{k}(x).
        \end{align*}
        \begin{proof} Let $G\subseteq\mathbb{R}^k$ be open and bounded, let $h\in\mathcal{C}^1(G; \mathbb{R}^{n-k})$, and let $K\subseteq G$ be a compact set such that $$\big\{\big(x, h(x)\big) : x\in K\}=S.$$ For the sake of brevity, set $$f(x)\coloneq \big(x, h(x)\big), \quad x\in G.$$ Hence $f(K)=S$ and, without loss of generality, we can assume that  $f$ is bi-Lipschitz on $G$.

        In this proof, for $x\in G$ we write $C_{f(x)}$ for
\[
P_{(\textup{T}^k_{f(x)}f(G))^\perp}(C).
\]
For $\mathcal H^k$-a.e. point of $S=f(K)$, this agrees with the notation
introduced above.

        Since $$\mathcal{H}^k(S)=\mathcal{H}^k\big(f(K)\big)\leq\textup{Lip}(f)^k\lambda^k(K)<\infty,$$ it suffices to show the conclusion for all convex bodies in $\mathcal{C}^{n, n}$ due to Lemma \ref{continuity} and the fact that $\mathcal{M}^k_{r, C}(S)\leq \mathcal{M}^k_{r, C^j}(S)$ for every $r>0$, where $C^j\coloneqq C\oplus B(0, j^{-1})$. So, let $C\in\mathcal{C}^{n, n}$. Choose positive real numbers $\alpha$ and $\beta$ such that \begin{equation}\label{constants}
            B(0, \alpha)\subseteq C \subseteq B(0, \beta).
        \end{equation}

            Fix some $\delta_0>0$ such that $K\oplus  B_{\mathbb{R}^k}(0, \delta_0)\subseteq G$. For every $p\in\mathbb{N}$, there exists some $\delta\in (0, \delta_0)$ such that 
            \begin{equation}\label{eq1}
                \lambda^k\Big(\big(K\oplus  B_{\mathbb{R}^k}(0, \delta)\big)\setminus K\Big)<\frac{1}{p}
            \end{equation} and for every $a\in K$ and every $x\in B(a, \delta)$ \begin{align}\label{eq2}
                |f(x)-f(a)-Df(a)(x-a)|\leq\frac{|x-a|}{p^2}.
            \end{align} and \begin{align}\label{eq00}
               |J_k{f(x)}-J_k{f(a)}|<\frac{1}{p},&&&&|\mathcal{H}^{n-k}(C_{f(x)})-\mathcal{H}^{n-k}(C_{f(a)})|<\frac{1}{p},
            \end{align} where we used the fact that the functions
$$
x\longmapsto J_kf(x)
\qquad\text{and}\qquad
x\longmapsto\mathcal H^{n-k}(C_{f(x)})
$$
are continuous on $G$ (the continuity of the latter follows from
Lemma~\ref{gong}), and hence uniformly continuous on the compact set
$K\oplus B_{\mathbb R^k}(0, \delta_0)$.

            Let $r\in(0, \min\{1, (\sqrt{k}p)^{-1}\delta \})$ and consider a tessellation of $\mathbb{R}^k$ into
 cubes of side length $pr$, and let $\mathcal{C}$ be the set of cubes from the tessellation
 hitting $K$. Then $ K\subseteq\bigcup \mathcal{C}$, and, taking \eqref{eq1} into account, \begin{align}\label{Eq4}
     \lambda^k\bigg(\Big(\bigcup\mathcal{C}\Big)\setminus K\bigg)<\frac{1}{p}.
 \end{align}

 For any cube $Q\in\mathcal{C}$, pick a point $a\in Q\cap K$ and consider the affine map $A_Q\colon\mathbb{R}^k\to\mathbb{R}^n$ given by $$A_Q(z)\coloneq f(a)+Df(a)(z-a), \quad\text{$z\in\mathbb{R}^k$.}$$ Observe that since $\rank Df(a)=k$, the function $A_Q$ is injective.

Note that if $z\in f(Q)\oplus rC$, then there exists some $y\in Q$ such that $z-f(y)\in rC$. Since $\diam(Q)\leq\sqrt{k}pr,$ \eqref{eq2} ensures that $$|f(y)-A_Q(y)|\leq\frac{\sqrt{k}r}{p}.$$  Thus \begin{equation}
    z-A_Q(y)\in rC \oplus B\big(0, \sqrt{k}rp^{-1}\big)\subseteq (1+\eta)r C
\end{equation}
with $\eta\coloneq\sqrt{k}(p\alpha)^{-1}$. It follows that \begin{equation}\label{Eq5}
    \lambda^n\big(f(Q)\oplus rC\big)\leq\lambda^n\big(A_Q(Q)\oplus (1+\eta)rC\big).
\end{equation} Since $\rank Df(x)=k$ at each point $x\in G$, the set $A_Q(Q)$ is a convex body in $\mathbb{R}^n$ of dimension exactly $k$.

Note that \begin{equation*}
    \diam\big(A_Q(Q)\big)\leq \textup{Lip}(f)\diam(Q)=\textup{Lip}(f)pr\sqrt{k}\eqqcolon t_r,
\end{equation*} which in turn implies \begin{equation}\label{pp4}
    A_Q(Q)\subseteq B(z_{Q}, t_r)=z_{Q}+B(0, t_r)
\end{equation} for some $z_{Q}\in\mathbb{R}^n$.

Using \eqref{pp4}, the monotonicity, translation invariance, and homogeneity properties in \eqref{pp5}, together with $\lambda^k(Q)=(pr)^k$ and $r\leq1$, we conclude  \newcommand\myeqhuhuj{\mathrel{\stackrel{\makebox[0pt]{\mbox{\normalfont\tiny (\ref{Eq5})}}}{\leq}}}\newcommand\myeqhuhujuj{\mathrel{\stackrel{\makebox[0pt]{\mbox{\normalfont\tiny (\ref{pp1})}}}{=}}}\newcommand\myeqhuhujujuj{\mathrel{\stackrel{\makebox[0pt]{\mbox{\normalfont\tiny (\ref{pp3})+(\ref{pp2})}}}{=}}}\newcommand\myeqhuhujujujuj{\mathrel{\stackrel{\makebox[0pt]{\mbox{\normalfont\tiny (\ref{constants})+(\ref{pp4})}}}{\leq}}}\begin{equation}\label{Eq1}
    \begin{split}
        \mathcal{M}^k_{r, C}\big(f(Q)\big)&\myeqhuhuj (1+\eta)^{n-k}\mathcal{M}^k_{(1+\eta)r, C}\big(A_Q(Q)\big)\\&\myeqhuhujuj\frac{1}{\omega_{n-k} r^{n-k}}\sum_{j=0}^k {n\choose j}V_n\big(A_Q(Q)[j], C[n-j]\big)r^{n-j}(1+\eta)^{n-j}\\&\myeqhuhujujuj\,\,\,\,\frac{(1+\eta)^{n-k}}{\omega_{n-k}}\int\limits_{A_Q(Q)}\mathcal{H}^{n-k}(C_x)\dH{k}(x)\\&\quad\quad+\frac{1}{\omega_{n-k}}\sum_{j=0}^{k-1} {n\choose j}V_n\big(A_Q(Q)[j], C[n-j]\big)r^{k-j}(1+\eta)^{n-j}\\&\myeqhuhujujujuj\,\,\,\,\,\frac{(1+\eta)^{n-k}}{\omega_{n-k}}\int\limits_{A_Q(Q)}\mathcal{H}^{n-k}(C_x)\dH{k}(x)\\&\quad\quad+\frac{1}{\omega_{n-k}}\sum_{j=0}^{k-1} {n\choose j}V_n\big(B(z_{Q}, t_r)[j], B(0, \beta)[n-j]\big)r^{k-j}(1+\eta)^{n-j}\\&\leq\frac{(1+\eta)^{n-k}}{\omega_{n-k}}\int\limits_{A_Q(Q)}\mathcal{H}^{n-k}(C_{f(a)})\dH{k}(x)\\&\quad\quad+\frac{L}{p\omega_{n-k}}(1+\eta)^{n}\lambda^k(Q),
    \end{split}
\end{equation} where $$L\coloneq\omega_{n}\sum_{j=0}^{k-1} {n\choose j}\textup{Lip}(f)^j\sqrt{k}^j\beta^{n-j}.$$

 The area formula and the injectivity of $A_Q$ imply
 \begin{equation*}
    \begin{split}
        \int\limits_{A_Q(Q)}\mathcal{H}^{n-k}(C_{f(a)})\dH{k}(x)=\int\limits_{Q}J_kA_Q(y)\mathcal{H}^{n-k}(C_{f(a)})\dlambda{k}(y).
        \end{split}
        \end{equation*}
        For every $x\in Q$, it holds that $$J_kA_Q(x)=J_kA_Q(a)=\sqrt{\det\big(Df(a)^{\top}Df(a)\big)}=J_kf(a),$$ which gives \newcommand\myeqhuhuju{\mathrel{\stackrel{\makebox[0pt]{\mbox{\normalfont\tiny (\ref{eq00})}}}{\leq}}}   \begin{equation}\label{Eq2}
    \begin{split}
        \int\limits_{Q}J_kA_Q(y)\mathcal{H}^{n-k}(C_{f(a)})\dlambda{k}(y)&=\int\limits_{Q}J_kf(a)\mathcal{H}^{n-k}(C_{f(a)})\dlambda{k}(x)\\&\myeqhuhuju\int\limits_{Q}\Big(J_kf(x)+\frac{1}{p}\Big)\Big(\mathcal{H}^{n-k}(C_{f(x)})+\frac{1}{p}\Big)\dlambda{k}(x)\\&\leq\int\limits_{Q}J_kf(x)\mathcal{H}^{n-k}(C_{f(x)})\dlambda{k}(x)\\&\quad\quad+\frac{\omega_{n-k}\diam(C)^{n-k}+M+1}{p}\lambda^k(Q), 
        \end{split}
        \end{equation}
where $$M\coloneq\sup\big\{|J_kf(x)| : x\in K\oplus B_{\mathbb{R}^k}(0, \delta_0)\big\}.$$
\medskip
\newcommand\myeqhuh{\mathrel{\stackrel{\makebox[0pt]{\mbox{\normalfont\tiny (\ref{Eq1})}}}{\leq}}}
\newcommand\myeqhu{\mathrel{\stackrel{\makebox[0pt]{\mbox{\normalfont\tiny (\ref{Eq2})}}}{\leq}}}
\newcommand\myeqh{\mathrel{\stackrel{\makebox[0pt]{\mbox{\normalfont\tiny (\ref{Eq4})}}}{\leq}}}

        Altogether, we have \begin{equation*}
            \begin{split}\mathcal{M}^k_{r, C}(S)&\leq\sum_{Q\in\mathcal{C}}\mathcal{M}^k_{r, C}\big(f(Q)\big)\\&\myeqhuh\frac{(1+\eta)^{n-k}}{\omega_{n-k}}\sum_{Q\in\mathcal{C}}\int\limits_{A_Q(Q)}\mathcal{H}^{n-k}(C_{f(a)})\dH{k}(x)+\frac{L}{p\omega_{n-k}}(1+\eta)^n\lambda^k\Big(\bigcup\mathcal{C}\Big)\\&\myeqhu\frac{(1+\eta)^{n-k}}{\omega_{n-k}}\int\limits_{\bigcup \mathcal{C}}J_kf(x)\mathcal{H}^{n-k}(C_{f(x)})\dlambda{k}(x)\\&\quad\quad+\frac{\omega_{n-k}\diam(C)^{n-k}+M+1+L}{p\omega_{n-k}}(1+\eta)^n\lambda^k\Big(\bigcup\mathcal{C}\Big)\\&\leq\frac{(1+\eta)^{n-k}}{\omega_{n-k}}\int\limits_{K}J_kf(x)\mathcal{H}^{n-k}(C_{f(x)})\dlambda{k}(x)\\&\quad\quad+M\diam(C)^{n-k}(1+\eta)^{n-k}\lambda^k\bigg(\Big(\bigcup\mathcal{C}\Big)\setminus K\bigg)\\&\quad\quad+\frac{\omega_{n-k}\diam(C)^{n-k}+M+1+L}{p\omega_{n-k}}(1+\eta)^n\lambda^k\Big(\bigcup\mathcal{C}\Big)\\&\myeqh\frac{(1+\eta)^{n-k}}{\omega_{n-k}}\int\limits_{K}J_kf(x)\mathcal{H}^{n-k}(C_{f(x)})\dlambda{k}(x)\\&\quad\quad+M\diam(C)^{n-k}\frac{(1+\eta)^{n-k}}{p}\\&\quad\quad+\frac{\omega_{n-k}\diam(C)^{n-k}+M+1+L}{p\omega_{n-k}}(1+\eta)^n\lambda^k\big(K\oplus B_{\mathbb{R}^k}(0, \delta_0)\big),
            \end{split}
        \end{equation*}
        from which it follows that \begin{equation*}
            \begin{split}
\mathcal{M}^k_{C}(S)^*&\leq\frac{(1+\eta)^{n-k}}{\omega_{n-k}}\int\limits_{K}J_kf(x)\mathcal{H}^{n-k}(C_{f(x)})\dlambda{k}(x)\\&\quad\quad+M\diam(C)^{n-k}\frac{(1+\eta)^{n-k}}{p}\\&\quad\quad+\frac{\omega_{n-k}\diam(C)^{n-k}+M+1+L}{p\omega_{n-k}}(1+\eta)^{n}\lambda^k\big(K\oplus B_{\mathbb{R}^k}(0, \delta_0)\big).
            \end{split}
        \end{equation*}
        Using area formula and the fact that $p\in\mathbb{N}$ was arbitrary, we obtain
     $$\mathcal{M}^k_{C}(S)^*\leq\frac{1}{\omega_{n-k}}\int\limits_{S}\mathcal{H}^{n-k}(C_{w})\dH{k}(w),$$ which completes the proof.
     \end{proof}
\end{lemma}

 Lemmata \ref{krecti} and \ref{krect} give us the following result. Later, we will see that the assumption $C\in\mathcal{C}^{n, n}$ can be relaxed to $C\in\mathcal{C}^{n}$.
\begin{lemma}\label{krectif}
    Let $C\in\mathcal{C}^{n, n}$, and let $S\subseteq\mathbb{R}^n$ be a compact subset of a $\mathcal{C}^1$ $k$-graph, where $k\in\{1, \dots, n-1\}$. Then we have \begin{align*}
            \mathcal{M}^k_{C}(S)=\frac{1}{\omega_{n-k}}
            \int\limits_{S}\mathcal{H}^{n-k}(C_x)\dH{k}(x).
        \end{align*}
\end{lemma}

\section{Anisotropic $k$-dimensional Minkowski Content for Countably $\mathcal{H}^k$-rectifiable Sets}\label{count}
\subsection{Lower Bound for General Countably $\mathcal{H}^k$-rectifiable Sets} Let us first establish a lower bound on $\mathcal{M}^k_C(S)_*$ for countably $\mathcal{H}^k$-rectifiable sets $S$. We will proceed similarly to the proof of \cite[Lemma 6]{Rataj5}.

\begin{lemma}\label{LP}
Let $T\in G_k(\mathbb R^n)$ and let $L\in G_m(\mathbb R^n)$ with
$k+m\geq n$. Let $C\subseteq L$ be a convex body. Then
\[
J^{n-k}(P_{T^\perp}|_L)\,
\mathcal H^{n-k}\big(P_{(T\cap L)^{\perp_L}}(C)\big)
=
\mathcal H^{n-k}\big(P_{T^\perp}(C)\big),
\]
where $(T\cap L)^{\perp_L}$ denotes the orthogonal complement of $T\cap L$
inside $L$.
\begin{proof}
Since
\[
\dim(T\cap L)\geq k+m-n,
\]
there are two cases.

First suppose that
\[
\dim(T\cap L)>k+m-n.
\]
Then
\[
\dim P_{T^\perp}(L)
=
m-\dim(T\cap L)
<
m-(k+m-n)
=
n-k.
\]
Since $C\subseteq L$, we have
\[
P_{T^\perp}(C)\subseteq P_{T^\perp}(L),
\]
and therefore
\[
\mathcal H^{n-k}\big(P_{T^\perp}(C)\big)=0.
\]
At the same time, the rank of the linear map
\[
P_{T^\perp}|_L\colon L\to T^\perp
\]
is strictly smaller than $n-k$, and hence
\[
J^{n-k}(P_{T^\perp}|_L)=0.
\]
Thus both sides vanish.

Now suppose that
\[
\dim(T\cap L)=k+m-n.
\]
Set
\[
W:=(T\cap L)^{\perp_L}.
\]
Then
\[
L=(T\cap L)\oplus W
\]
orthogonally inside $L$, and
\[
\dim W
=
m-\dim(T\cap L)
=
m-(k+m-n)
=
n-k.
\]

Consider the restriction
\[
A:=P_{T^\perp}|_W\colon W\to T^\perp.
\]
We claim that $A$ is injective. Indeed, if $w\in W$ and $A(w)=0$, then
$P_{T^\perp}(w)=0$, and hence $w\in T$. Since also $w\in W\subseteq L$, we
have
\[
w\in T\cap L.
\]
But $w\in W=(T\cap L)^{\perp_L}$, so $w=0$. Thus $A$ is injective. Since
\[
\dim W=\dim T^\perp=n-k,
\]
the map $A$ is a linear isomorphism from $W$ onto $T^\perp$.

For every $c\in C\subseteq L$, write uniquely
\[
c=t+w,
\qquad
t\in T\cap L,\quad w\in W.
\]
Then
\[
P_W(c)=w
\]
and, since $t\in T$,
\[
P_{T^\perp}(c)
=
P_{T^\perp}(t+w)
=
P_{T^\perp}(w)
=
A(w).
\]
Therefore
\[
P_{T^\perp}(C)
=
A\big(P_W(C)\big)
=
A\big(P_{(T\cap L)^{\perp_L}}(C)\big).
\]
By the area formula for the linear isomorphism $A$,
\[
\mathcal H^{n-k}\big(P_{T^\perp}(C)\big)
=
J^{n-k}(A)\,
\mathcal H^{n-k}\big(P_{(T\cap L)^{\perp_L}}(C)\big).
\]

It remains to identify $J^{n-k}(A)$ with
$J^{n-k}(P_{T^\perp}|_L)$. With respect to the orthogonal decomposition
\[
L=(T\cap L)\oplus W,
\]
the map $P_{T^\perp}|_L$ vanishes on $T\cap L$ and coincides with $A$ on
$W$. Hence the nonzero singular values of $P_{T^\perp}|_L$ are precisely
the singular values of $A$. Therefore
\[
J^{n-k}(P_{T^\perp}|_L)=J^{n-k}(A).
\]
Combining this with the previous identity gives
\[
J^{n-k}(P_{T^\perp}|_L)\,
\mathcal H^{n-k}\big(P_{(T\cap L)^{\perp_L}}(C)\big)
=
\mathcal H^{n-k}\big(P_{T^\perp}(C)\big),
\]
as required.
\end{proof}
\end{lemma}

\begin{proposition}[Lower bound]\label{LB}
Let $S\subseteq\mathbb{R}^n$ be a countably $\mathcal{H}^k$-rectifiable set,
where $k\in\{1,\dots,n-1\}$, and let $C\in\mathcal{C}^n$. Then
\[
\mathcal{M}^k_{C}(S)_*
\geq
\frac{1}{\omega_{n-k}}
\int\limits_{S}\mathcal{H}^{n-k}(C_x)\,\dH{k}(x).
\]
\begin{proof}
Since $S$ is countably
$\mathcal H^k$-rectifiable, there exists a sequence
$\{S_i\}_{i=1}^{\infty}$ of pairwise disjoint compact subsets of $S$
covering $S$ up to an $\mathcal H^k$-negligible set, such that each $S_i$
is contained in a $\mathcal C^1$ $k$-graph. Set $$S^{\prime}\coloneqq\bigcup_{i=1}^{\infty}S_i.$$ Then $S^{\prime}$ is Borel and $\mathcal{H}^k(S\setminus S^{\prime})=0$.

First, assume that $C\in\mathcal{C}^{n, n}$. 

Let $N\in\mathbb N$ be arbitrary. Since $S_1,\dots,S_N$ are pairwise
disjoint compact sets, their mutual distances are positive. Hence, for all
sufficiently small $r>0$, the sets
\[
S_i\oplus rC,\qquad i\in\{1, \dots, N\},
\]
are pairwise disjoint. Therefore, by Lemma~\ref{krectif},
\[
\begin{split}
\sum_{i=1}^{N}
\frac{1}{\omega_{n-k}}
\int\limits_{S_i}\mathcal H^{n-k}(C_x)\,\dH{k}(x)
&=
\sum_{i=1}^{N}\mathcal M_C^k(S_i)
\\
&=
\lim_{r\to0_+}
\frac{1}{\omega_{n-k}r^{n-k}}
\lambda^n\Big(\bigcup_{i=1}^{N}(S_i\oplus rC)\Big)
\\
&\leq
\liminf_{r\to0_+}
\frac{\lambda^n(S^{\prime}\oplus rC)}{\omega_{n-k}r^{n-k}}
\\
&=
\mathcal M_C^k(S^{\prime})_*\\
&\leq
\mathcal M_C^k(S)_*.
\end{split}
\]
Letting $N\to\infty$ and taking Remark \ref{consistency} into account gives
\[
\frac{1}{\omega_{n-k}}
\int\limits_{S^{\prime}}\mathcal H^{n-k}(C_x)\,\dH{k}(x)=\frac{1}{\omega_{n-k}}
\int\limits_{S}\mathcal H^{n-k}(C_x)\,\dH{k}(x)
\leq
\mathcal M_C^k(S)_*.
\]
This proves the desired lower bound when $C$ is full-dimensional.

Now assume that $\dim(C)=m<n$ and denote by
\[
L:=\operatorname{span}(C)
\]
the linear span of $C$. Since $S^{\prime}$ is Borel, the sets $$S^{\prime}\oplus rC, \qquad r>0,$$ are measurable (see Remark \ref{measurability}).

If $k+m<n$, then for $\mathcal H^k$-a.e. $x\in S^{\prime}$ we have
\[
\mathcal H^{n-k}(C_x)
=
\mathcal H^{n-k}\big(P_{\textup N_x^{n-k}S^{\prime}}(C)\big)=0.
\]
Indeed, $P_{\textup N_x^{n-k}S^{\prime}}(C)$ has dimension at most $m<n-k$.
Consequently, taking Remark \ref{consistency} into account
\[
\frac{1}{\omega_{n-k}}
\int\limits_{S^{\prime}}\mathcal H^{n-k}(C_x)\,\dH{k}(x)=\frac{1}{\omega_{n-k}}
\int\limits_{S}\mathcal H^{n-k}(C_x)\,\dH{k}(x)=0
\leq
\mathcal M_C^k(S^{\prime})_*\leq\mathcal M_C^k(S)_*.
\]

It remains to consider the case $k+m\geq n$. Since $C\subseteq L$, Fubini's
theorem gives
\[
\lambda^n(S^{\prime}\oplus rC)
=
\int\limits_{L^\perp}
\mathcal H^m\Big(\big(S^{\prime}\cap(x+L)\big)\oplus rC\Big)\,\dH{n-m}(x).
\]
By Theorem~\ref{rezy}, the set $S^{\prime}\cap(x+L)$ is countably
$\mathcal H^{k+m-n}$-rectifiable for $\mathcal H^{n-m}$-a.e.
$x\in L^\perp$. Fatou's lemma, the lower bound already proved in the
full-dimensional case applied inside the space $L$ (if \(k+m-n=0\), we instead use formula \eqref{0n}), and the coarea formula
for the projection $P_{L^\perp}$ yield
\[
\begin{split}
\mathcal M_C^k(S^{\prime})_*
&\geq
\int\limits_{L^\perp}
\liminf_{r\to0_+}
\frac{
\mathcal H^m\Big(\big(S^{\prime}\cap(x+L)\big)\oplus rC\Big)}
{\omega_{n-k}r^{m-(k+m-n)}}
\,\dH{n-m}(x)
\\
&\geq
\frac{1}{\omega_{n-k}}
\int\limits_{L^\perp}
\int\limits_{S^{\prime}\cap(x+L)}
\mathcal H^{n-k}
\big(P_{(\textup T_y^kS^{\prime}\cap L)^{\perp_L}}(C)\big)
\,\dH{k+m-n}(y)\,\dH{n-m}(x)
\\
&=
\frac{1}{\omega_{n-k}}
\int\limits_{S^{\prime}}
J^{n-m}\big(P_{L^\perp}|_{\textup T_w^kS^{\prime}}\big)
\mathcal H^{n-k}
\big(P_{(\textup T_w^kS^{\prime}\cap L)^{\perp_L}}(C)\big)
\,\dH{k}(w)
\\
&=
\frac{1}{\omega_{n-k}}
\int\limits_{S^{\prime}}
J^{n-k}\big(P_{\textup N_w^{n-k}S^{\prime}}|_L\big)
\mathcal H^{n-k}
\big(P_{(\textup T_w^kS^{\prime}\cap L)^{\perp_L}}(C)\big)
\,\dH{k}(w).
\end{split}
\]
Here the normal space to the slice is understood inside the affine space
$x+L$, and hence equals
\[
(\textup T_w^kS^{\prime}\cap L)^{\perp_L}.
\]
In the non-transversal case the Jacobian
\[
J^{n-k}\big(P_{\textup N_w^{n-k}S^{\prime}}|_L\big)
\]
vanishes, while in the transversal case the tangent space to the slice is
\[
\textup T_w^{k+m-n}\big(S^{\prime}\cap(P_{L^\perp}(w)+L)\big)
=
\textup T_w^kS^{\prime}\cap L.
\]

Applying Lemma~\ref{LP} with
\[
T=\textup T_w^kS^{\prime}
\]
gives, for $\mathcal H^k$-a.e. $w\in S^{\prime}$,
\[
J^{n-k}\big(P_{\textup N_w^{n-k}S^{\prime}}|_L\big)
\mathcal H^{n-k}
\big(P_{(\textup T_w^kS^{\prime}\cap L)^{\perp_L}}(C)\big)
=
\mathcal H^{n-k}\big(P_{\textup N_w^{n-k}S^{\prime}}(C)\big).
\]
Since
\[
P_{\textup N_w^{n-k}S^{\prime}}(C)=C_w,
\]
we conclude, taking Remark \ref{consistency} into account once again, that
\[
\mathcal M_C^k(S)_*\geq \mathcal M_C^k(S^{\prime})_*
\geq
\frac{1}{\omega_{n-k}}
\int\limits_{S^{\prime}}\mathcal H^{n-k}(C_w)\,\dH{k}(w)=\frac{1}{\omega_{n-k}}
\int\limits_{S}\mathcal H^{n-k}(C_w)\,\dH{k}(w),
\]
which completes the proof.
\end{proof}
\end{proposition}

In general, to obtain a reasonable upper bound on $\mathcal{M}^k_C(S)^*$, one needs to impose additional assumptions, as we will see in the next subsection.

\subsection{The Family of AFP-conditions}
\label{SecDef}
In this section, we introduce a family of conditions commonly referred to as \emph{AFP-conditions}. The name reflects the initials of Ambrosio, Fusco and Pallara, in analogy with their foundational work \cite[p.~110]{ambrosio}, where these conditions are discussed.
\begin{definition}[AFP-$k$-condition]
    Let $S$ be a closed subset of $\mathbb{R}^n$, and let $k\in\{1, \dots, n-1\}.$ We say that $S$ satisfies the \emph{AFP-$k$-condition} provided that there exist a constant $\gamma>0$ and Radon measure $\mu$, absolutely continuous with respect to $\mathcal{H}^k$, such that \begin{equation}\label{pp21}
        \mu\big(B(x, r)\big)\geq\gamma r^{k},\quad\text{$x\in S$, $r\in(0, 1)$.}
    \end{equation}
\end{definition}
The following remark gives an example of a class of sets that satisfy the AFP-$k$-condition.
\begin{rem}[\protect{\cite[Remark 1]{Villa4}}]\label{rema}
    If $S\subseteq\mathbb{R}^n$ is a $k$-rectifiable compact set, where $k\in\{1, \dots, n-1\}$, then $S$ satisfies the AFP-$k$-condition.
\end{rem}

The validity of the AFP-
$k$-condition guarantees that, for a certain class of sets, the isotropic 
$k$-dimensional Minkowski content coincides with the 
$k$-dimensional Hausdorff measure.

\begin{theorem}[\protect{\cite[Theorem 2.104]{ambrosio}}]\label{pp14}
    If $S$ is a countably $\mathcal{H}^k$-rectifiable compact set and the AFP-$k$-condition holds for $S$, then $$\mathcal{M}^k_{B(0, 1)}(S)=\mathcal{H}^k(S).$$
\end{theorem}

In the anisotropic setting, only the case $k=n-1$ is well understood.
\begin{theorem}[\protect{\cite[Theorem 3.4]{Villa2}}]\label{pp13}
    If $S$ is a countably $\mathcal{H}^{n-1}$-rectifiable compact set and the AFP-$(n-1)$-condition holds for $S$, then \begin{equation}\label{Eq7}
        \mathcal{M}^{n-1}_{C}(S)=\frac{1}{2}\int\limits_{S}\big(h_C(\nu_S)+h_C(-\nu_S)\big)\dH{n-1}
    \end{equation} for any $C\in\mathcal{C}^{n, n}.$
\end{theorem}
Let us mention that the right-hand side of \eqref{Eq7} is in the case $k=n-1$ exactly \eqref{limit}:
\begin{lemma}\label{h_C}
    Let $C\in\mathcal{C}^{n, n}$, and let $L$ be an $(n-1)$-dimensional linear subspace of $\mathbb{R}^n$. Then $$\rho_{P_{L^{\perp}}(C)}(\nu)=h_C(\nu)\qquad\text{for every $\nu\in L^{\perp}\cap\mathbb{S}^{n-1}.$}$$
    \begin{proof}
        Fix a $\nu\in L^{\perp}\cap\mathbb{S}^{n-1}$. Then $$C\subseteq\big\{y\in\mathbb{R}^n: y\cdot\nu\leq h_C(\nu)\big\}.$$ Thus, $$P_{L^{\perp}}(C)\subseteq\big\{y\in L^{\perp}: y\cdot\nu\leq h_C(\nu)\big\}.$$ It holds that $\rho_{P_{L^{\perp}}(C)}(\nu)\nu\in P_{L^{\perp}}(C)$, whence $$h_C(\nu)\geq\big(\rho_{P_{L^{\perp}}(C)}(\nu)\nu\big)\cdot\nu=\rho_{P_{L^{\perp}}(C)}(\nu).$$

On the other hand, fix an $x\in\big\{y\in\mathbb{R}^n: y\cdot\nu= h_C(\nu)\big\}\cap C.$ Then $$P_{L^{\perp}}(x)\in P_{L^{\perp}}\big(\big\{y\in\mathbb{R}^n: y\cdot\nu= h_C(\nu)\big\}\big)\cap P_{L^{\perp}}(C),$$ from which it follows $$h_C(\nu)\nu=P_{L^{\perp}}(x)\in \big\{y\in L^{\perp}: y\cdot\nu= h_C(\nu)\big\}\cap P_{L^{\perp}}(C).$$
        Hence 
        $$h_C(\nu)\leq\max\big\{\lambda\geq 0:\lambda\nu\in P_{L^{\perp}}(C)\big\}=\rho_{P_{L^{\perp}}(C)}(\nu),$$ from which the conclusion follows.
    \end{proof}
\end{lemma}
For $k=n-1$, we have $\textup{n}^{n-k}_{x}S=\{\nu_S(x), -\nu_{S}(x)\}$ for $\mathcal{H}^{n-1}$-a.a. $x\in S$.  In view of Lemma \ref{h_C}, we conclude that for any $C\in\mathcal{C}^{n, n}$\begin{equation*}
    \begin{split} \int\limits_{S}\fint\limits_{\textup{n}^{n-k}_{x}S}\rho_{C_x}(\nu)\dH{0}(\nu)\dH{n-1}(x)&=\int\limits_{S}\fint\limits_{\textup{n}^{n-k}_{x}S}h_C(\nu)\dH{0}(\nu)\dH{n-1}(x)\\&=\frac{1}{2}\int\limits_{S}\big(h_C(\nu_{S})+h_C(-\nu_{S})\big)\dH{n-1}.
    \end{split}
\end{equation*}

\begin{rem}\label{pp15}
Note that if the set $S\subseteq\mathbb{R}^n$ satisfies the AFP-$k$-condition for some $k\in\{1, \dots, n-1\}$ and some $\mu$ with $\mu(\mathbb{R}^n)<\infty$, then $S$ is totally bounded and $\mathcal{H}^k(S)<\infty$. Indeed, let $\rho\in (0, 1)$ and let $S(\rho)\subseteq S$ be a maximal strictly $\rho$-separated set (that is, for every $x, y\in S(\rho)$ such that $x\neq y$, it holds that $|x-y|> \rho$), then \begin{align*}
    S\subseteq\bigcup_{x\in S(\rho)}B(x, \rho)
\end{align*} and the balls $\{B(x, \rho/2)\}_{x\in S(\rho)}$ are pairwise disjoint. Consequently, $$\infty>\mu(\mathbb{R}^n)\geq\mu\Big(\bigcup_{x\in S(\rho)} B(x, \rho/2)\Big)=\sum_{x\in S(\rho)}\mu\big(B(x, \rho/2)\big)\geq\gamma\mathcal{H}^{0}\big(S(\rho)\big)(\rho/2)^k,$$ which in turn implies that $\mathcal{H}^{0}\big(S(\rho)\big)<\infty$, confirming the total boundedness of $S$. 

Furthermore, $S$ is covered by the balls $\{B(x, \rho)\}_{x\in S(\rho)}$ and the sum of their diameters satisfies $$\sum_{x\in S(\rho)}(2\rho)^k=4^k\mathcal{H}^{0}\big(S(\rho)\big)(\rho/2)^k\leq\frac{4^k}{\gamma}\mu(\mathbb{R}^n)<\infty.$$ Since $\rho>0$ is arbitrary, we get $\mathcal{H}^k(S)<\infty.$

On the other hand, if $S$ is compact and satisfies the AFP-$k$-condition for some $k\in\{1, \dots, n-1\}$, then the measure $\mu$ can be chosen finite. Indeed, we can replace $\mu$ by its restriction to $S\oplus B(0, 1)$.

Thus, the compactness of $S$ in Theorems \ref{pp14} and \ref{pp13} can be replaced by the closedness of $S$ and finiteness of $\mu.$ 

The closedness of $S$ cannot be omitted, as the following example illustrates.
\end{rem}

\begin{example}
    Let $S_1\coloneqq [0, 1]\times\{0\}\subseteq\mathbb{R}^2$ and $S_2\coloneqq S_1\cap (\mathbb{Q}\times \{0\})\subseteq\mathbb{R}^2$. Then both $S_1$ and $S_2$ are countably $\mathcal{H}^1$-rectifiable and satisfy the inequality \eqref{pp21} from the AFP-$1$-condition with $\mu\coloneqq\mathcal{H}^1\big|_{S_1}$. Furthermore, $\overline{S_2}=S_1$. Theorem \ref{pp14} and Remark \ref{pp11} imply $$\mathcal{M}^1_{B(0,1)}(S_2)=\mathcal{M}^1_{B(0,1)}(S_1)=\mathcal{H}^1(S_1)=1.$$ However, $\mathcal{H}^1(S_2)=0.$ 
\end{example}
\medskip

In a recent paper, Kiderlen and Rataj \cite{Rataj5} introduced a new variant of the AFP-condition with respect to a given linear subspace $L$.
\begin{definition}[AFP-condition relative to $L$]
    Let $S$ be a closed subset of $\mathbb{R}^n$, and let $L$ be an $m$-dimensional subspace of $\mathbb{R}^n$. We say that $S$ satisfies the \emph{AFP-condition relative to $L$} provided that there exist a constant $\gamma>0$ and Radon measure $\mu$, absolutely continuous with respect to $\mathcal{H}^{n-1}$, such that $$\mu\big(B(x, r)\big)\geq\gamma r^{m-1}\mathcal{H}^{n-m}\Big(P_{L^{\perp}}\big(S\cap B(x, r)\big)\Big),\qquad x\in S, r\in(0, 1).$$
\end{definition}
Note that in the above definition, the projection $P_{L^{\perp}}\big(S\cap B(x, r)\big)$ is compact, and hence $\mathcal{H}^{n-m}$-measurable.

\medskip

    The following theorem extends the result of Lussardi and Villa \cite[Theorem 3.4]{Villa2}. Namely, Kiderlen and Rataj confirmed that for countably $\mathcal{H}^{n-1}$-rectifiable compact sets, the validity of the AFP-condition relative to $L$ ensures the existence of the $C$-anisotropic $(n-1)$-dimensional Minkowski content for any $C\in\mathcal{C}^n$ with $C\subseteq L$.

\begin{theorem}[\protect{\cite[Theorem 8]{Rataj5}}]
    If $S$ is a countably $\mathcal{H}^{n-1}$-rectifiable compact set with finite $\mathcal{H}^{n-1}$-measure and $L$ is an $m$-dimensional subspace of $\mathbb{R}^n$ such that the AFP-condition relative to $L$ holds for $S$, then $$\mathcal{M}^{n-1}_{C}(S)=\frac{1}{2}\int\limits_{S}\big(h_C(\nu_S)+h_C(-\nu_S)\big)\dH{n-1}$$ for any $C\in\mathcal{C}^n$ with $C\subseteq L$.
\end{theorem}
Let us introduce the \emph{AFP-$k$-condition relative to $L$} which combines the previous two ideas.

\begin{definition}[AFP-$k$-condition relative to $L$]
    Let $S$ be a closed subset of $\mathbb{R}^n$, and let $L$ be an $m$-dimensional subspace of $\mathbb{R}^n$. We say that the set $S$ satisfies the \emph{AFP-$k$-condition relative to $L$} provided that there exist a constant $\gamma>0$ and Radon measure $\mu$, absolutely continuous with respect to $\mathcal{H}^k$, such that $$\mu\big(B(x, r)\big)\geq\gamma r^{k+m-n}\mathcal{H}^{n-m}\Big(P_{L^{\perp}}\big(S\cap B(x, r)\big)\Big),\quad\text{$x\in S$, $r\in(0, 1)$.}$$
\end{definition}

Our current objective is to prove that if the AFP-$k$-condition relative to an $m$-dimensional subspace $L$ is satisfied by a countably $\mathcal{H}^k$-rectifiable compact set $S$ with finite $\mathcal{H}^k$-measure, where $m+k-n>0$, then for any $C\in\mathcal{C}^n$ satisfying $C\subseteq L$, we have $$\mathcal{M}^k_{C}(S)=\frac{1}{\omega_{n-k}}
            \int\limits_{S}\mathcal{H}^{n-k}(C_x)\dH{k}(x).$$
            However, we prove something more general: We replace the compactness of $S$ by the 
            closedness of $S$ and finiteness of $\mu$. This gives us the result for a broader class of sets, since if $m<n$, then the AFP-$k$-condition relative to $L$ together with $\mu(\mathbb{R}^n)<\infty$ does not imply the total boundedness of $S$, as the following example shows.

\begin{example}
    Let $S\coloneqq\mathbb{R}^2\times\{0\}\subseteq\mathbb{R}^3$, and let $L\coloneqq S$. Then $S$ is countably $\mathcal{H}^2$-rectifiable,  and for every $x\in S$, we have $$\mathcal{H}^{1}\Big(P_{L^{\perp}}\big(S\cap B(x, r)\big)\Big)=\mathcal{H}^{1}(\{0\})=0,\qquad r>0.$$ So, we may choose $\mu\coloneqq0$. However, $S$ is not bounded.
\end{example}

We first treat the case \(m=n\). To confirm this, we use the following result:
\begin{theorem}[\protect{\cite[Theorem 4.4]{frys2}}]\label{s}
     Let $S$ be compact, $C\in\mathcal{C}^{n, n}$ and $s\in[0, n]$. Suppose that there exists a countable pairwise disjoint family $\{S_k\}_{k\in\mathbb{N}}$ of compact subsets of $S$ such that $\mathcal{M}_C^s(S_k)$ exists for any $k\in\mathbb{N}$ and  $$\mathcal{H}^s\Big(S\setminus\bigcup_{k\in\mathbb{N}}S_k\Big)=0,$$ and there is a positive Radon measure $\mu$ on $\mathbb{R}^n$ which is absolutely continuous with respect to $\mathcal{H}^s$, so that for some $\gamma>0$ and for each $x\in S$ and any $t\in(0, 1)$ we have $$\mu\big(B(x, t)\big)\geq\gamma t^{s}. $$ Then $\mathcal{M}_C^s(S)$ exists and $$\mathcal{M}_C^s(S)=\sum_{k\in\mathbb{N}}\mathcal{M}_C^s(S_k).$$
\end{theorem}

Note that the compactness of $S$ in Theorem \ref{s} can be replaced with the closedness of $S$ and finiteness of $\mu$, similarly as in Remark \ref{pp15}. 

\begin{theorem}\label{sq}
    Let $C\in\mathcal{C}^{n}$ be arbitrary, and let $S\subseteq\mathbb{R}^n$ be a countably $\mathcal{H}^k$-rectifiable closed set satisfying the AFP-$k$-condition for some finite Radon measure $\mu$, where $k\in\{1, \dots, n-1\}.$ Then $$\mathcal{M}^k_{C}(S)=\frac{1}{\omega_{n-k}}
            \int\limits_{S}\mathcal{H}^{n-k}(C_x)\dH{k}(x).$$
    \begin{proof}
        If $\dim(C)=n,$ the claim follows directly from Remark \ref{remark}, Lemma \ref{krectif} and Theorem~\ref{s}.
        
If $\dim(C)<n$, we construct a sequence of full-dimensional convex bodies 
$\{C^i\}_{i\in\mathbb{N}}$ such that $C^i \supseteq C$ for all $i\in\mathbb{N}$ and 
$C^i \xrightarrow[i\to\infty]{d_H} C$.

Using Proposition~\ref{LB} and the already established equality for full-dimensional 
convex bodies, we obtain the following chain of inequalities:
\begin{equation*}
    \begin{split}
        \frac{1}{\omega_{n-k}}
            \int\limits_{S}\mathcal{H}^{n-k}(C_x)\dH{k}(x)
        &\leq \mathcal{M}^k_C(S)_*
        \leq \mathcal{M}^k_C(S)^*
        \leq \mathcal{M}^k_{C^i}(S)\\&= \frac{1}{\omega_{n-k}}
            \int\limits_{S}\mathcal{H}^{n-k}(C^i_x)\dH{k}(x),
    \end{split}
\end{equation*}
which, together with Lemma~\ref{continuity} and the fact that $\mathcal{H}^k(S)<\infty$, following from Remark \ref{pp15}, implies the conclusion.
    \end{proof}
\end{theorem}
 
As a consequence of the previous theorem and Remark \ref{rema}, we obtain the following result.
\begin{theorem}
    Let $C\in\mathcal{C}^n$ be arbitrary, and let $S\subseteq\mathbb{R}^n$ be a $k$-rectifiable compact set, where $k\in\{1, \dots, n-1\}.$ Then $$\mathcal{M}^k_{C}(S)=\frac{1}{\omega_{n-k}}
            \int\limits_{S}\mathcal{H}^{n-k}(C_x)\dH{k}(x).$$
\end{theorem}

\medskip

For an arbitrary lower-dimensional structuring element $C\in\mathcal{C}^n$ and for a countably $\mathcal{H}^k$-rectifiable closed set $S$, we proceed analogously to the proof of \cite[Theorem 8]{Rataj5}. 

\subsection{The Case $k+m-n>0$}
The proof of the following theorem is inspired by \cite[Theorem 8]{Rataj5}.
\begin{theorem}
        Let $S\subseteq\mathbb{R}^n$ be a closed countably $\mathcal{H}^k$-rectifiable set with finite $\mathcal{H}^k$-measure satisfying the AFP-$k$-condition relative to an $m$-dimensional linear subspace $L\subseteq\mathbb{R}^n$ with finite measure $\mu$, where $k\in\{1, \dots, n-1\}$ and $k+m-n>0$, that is, there exist a constant $\gamma>0$ and a finite Radon measure $\mu$, absolutely continuous with respect to $\mathcal{H}^k$, such that $$\mu\big(B(x, r)\big)\geq\gamma r^{k+m-n}\mathcal{H}^{n-m}\Big(P_{L^{\perp}}\big(S\cap B(x, r)\big)\Big),\quad\text{$x\in S$, $r\in (0, 1)$.}$$ Then, for any $C\in\mathcal{C}^n$ with $C\subseteq L$, we have \begin{align*}
            \mathcal{M}^k_{C}(S)=\frac{1}{\omega_{n-k}}
            \int\limits_{S}\mathcal{H}^{n-k}(C_x)\dH{k}(x).
        \end{align*}
        \begin{proof}
            Due to Proposition \ref{LB}, it suffices to show $$\mathcal{M}^k_{C}(S)^*\leq\frac{1}{\omega_{n-k}}
            \int\limits_{S}\mathcal{H}^{n-k}(C_x)\dH{k}(x).$$
            
            Assume that the dimension of $C$ is exactly $m$. Further, choose positive constants $\alpha$ and $\beta$ such that $B_L(0, \alpha)\subseteq C\cap (-C)$ and $C\subseteq B_L(0, \beta).$

            Let $\varepsilon\in (0, 1)$ and set $$C^{\varepsilon}\coloneq C\oplus B_{L^{\perp}}(0, \alpha\varepsilon).$$

            Let $\{S_i\}_{i=1}^{\infty}$ be a sequence of pairwise disjoint compact
subsets of $S$ such that each $S_i$ is contained in a $\mathcal C^1$
$k$-graph and
\[
\mathcal H^k\Big(S\setminus\bigcup_{i=1}^{\infty}S_i\Big)=0.
\]
The absolute continuity of $\mu$ with respect to $\mathcal{H}^k$ ensures that
\[
\mu\Big(S\setminus\bigcup_{i=1}^{\infty}S_i\Big)=0.
\] 

By the continuity from above and the finiteness of $\mu$, we have that for this $\varepsilon>0$, there exists some $M\in\mathbb{N}$ such that $$\mu\Big(S\setminus\bigcup_{i=1}^{M}S_i\Big)<\varepsilon^{k+m-n+1}.$$

Set $$\kappa\coloneqq\varepsilon^{1/(k+m-n+1)}$$, 
\[
A(r)\coloneqq
\bigl(S\oplus(-\kappa rC^\varepsilon)\bigr)
\setminus\bigcup_{i=1}^M S_i, \qquad r>0.
\] and $$S_{\varepsilon, r}\coloneq S\setminus \Big(\bigcup_{i=1}^{M}S_i\oplus \kappa rC^{\varepsilon}\Big),\qquad r>0.$$ The sets \(A(r)\) decrease as \(r\to0_+\). Since \(S\) is closed and
\(C^\varepsilon\) contains the origin in its interior,
\[
\bigcap_{r>0}A(r)
=
S\setminus\bigcup_{i=1}^M S_i.
\] The continuity from above
of the finite measure \(\mu\) implies that there exists $0<r_{\varepsilon}<(\varepsilon\alpha\kappa)^{-1}$ such that $$\mu\Big(\big(S\oplus (-\kappa r_{\varepsilon}C^{\varepsilon})\big)\setminus \bigcup_{i=1}^M S_i\Big)<\varepsilon^{k+m-n+1}.$$ 

Consequently, $$\mu\Big(\big(S\oplus (-\kappa rC^{\varepsilon})\big)\setminus \bigcup_{i=1}^M S_i\Big)<\varepsilon^{k+m-n+1},\qquad 0<r<r_{\varepsilon}.$$

      By Lemma \ref{bes}, there exists an at most countable set $I_{\varepsilon, r}\subseteq S_{\varepsilon, r}$ such that \begin{align}\label{N}
              S_{\varepsilon, r}\subseteq\bigcup_{x\in I_{\varepsilon, r}} B(x, \varepsilon\alpha\kappa r)&&\text{and}&&\sum_{x\in I_{\varepsilon, r}}\chi_{B(x, \varepsilon\alpha\kappa r)}\leq 3^n.
      \end{align} 
      
        The AFP-$k$-condition relative to $L$ and \eqref{N} give us \begin{equation}\label{eq7}
            \begin{split}
                \sum_{x\in I_{\varepsilon, r}}&\gamma(\varepsilon\alpha\kappa r)^{k+m-n}\mathcal{H}^{n-m}\Big(P_{L^{\perp}}\big(S\cap B(x, \varepsilon\alpha\kappa r)\big)\Big)\\&\leq\sum_{x\in I_{\varepsilon, r}}\mu\big(B(x, \varepsilon\alpha\kappa r)\big)\\&\leq 3^n\mu\Big(\bigcup_{x\in I_{\varepsilon, r}}B(x, \varepsilon\alpha\kappa r)\Big).\end{split}
        \end{equation}
        Observe $$\big(S_{\varepsilon, r}\oplus (-\kappa r C^{\varepsilon})\big)\cap\bigcup_{i=1}^M S_i=\emptyset.$$ Indeed, if $x\in s -\kappa r C^{\varepsilon}$ for some $s\in S_{\varepsilon, r}$ and, at the same time, $x\in\bigcup_{i=1}^M S_i$, we have $s\in \bigcup_{i=1}^M S_i\oplus \kappa r C^{\varepsilon}$, which contradicts the fact that $s\in S_{\varepsilon, r}.$

        Furthermore, since it holds that $$B(0, \alpha)\subseteq B_L(0, \alpha)\oplus B_{L^{\perp}}(0, \alpha)\subseteq-\frac{1}{\varepsilon}C\oplus \frac{1}{\varepsilon}B_{L^{\perp}}(0, \varepsilon\alpha)\subseteq -\frac{1}{\varepsilon}C^{\varepsilon},$$
we have \begin{equation*}
    \begin{split}
        \bigcup_{x\in I_{\varepsilon, r}}B(x, \varepsilon\alpha\kappa r)&\subseteq \big(S_{\varepsilon, r}\oplus (-\kappa rC^{\varepsilon})\big)\setminus\bigcup_{i=1}^M S_i\\&\subseteq \big(S\oplus (-\kappa rC^{\varepsilon})\big)\setminus\bigcup_{i=1}^M S_i,
    \end{split}
\end{equation*}
which coupled with \eqref{eq7} yields \begin{equation*}\begin{split}
\sum_{x\in I_{\varepsilon, r}}&\gamma(\varepsilon\alpha\kappa r)^{k+m-n}\mathcal{H}^{n-m}\Big(P_{L^{\perp}}\big(S\cap B(x, \varepsilon\alpha\kappa r)\big)\Big)\\&\leq 3^n\mu\Big(\big(S\oplus (-\kappa r C^{\varepsilon})\big)\setminus\bigcup_{i=1}^M S_i\bigg)\\&<\varepsilon^{k+m-n+1} 3^n,\end{split}
\end{equation*} from which it follows \begin{equation}\label{odhad}
    \sum_{x\in I_{\varepsilon, r}}\mathcal{H}^{n-m}\Big(P_{L^{\perp}}\big(S\cap B(x, \varepsilon\alpha\kappa r)\big)\Big)\leq \frac{\varepsilon 3^n}{\gamma(\alpha\kappa r)^{k+m-n}}=\frac{\varepsilon^{1/(k+m-n+1)} 3^n}{\gamma\alpha^{k+m-n} r^{k+m-n}}
\end{equation}

        Using \eqref{N} and \eqref{odhad}, we can estimate \begin{equation*}
            \begin{split}
                \lambda^n(S_{\varepsilon, r}\oplus rC)&\leq\lambda^n\big(S_{\varepsilon, r}\oplus B_{L}(0, \beta r)\big)\\&\leq \sum_{x\in I_{\varepsilon, r}}\lambda^n\Big(\big(S\cap B(x, \varepsilon\alpha\kappa r)\big)\oplus B_{L}(0, \beta r)\Big)\\&\leq \sum_{x\in I_{\varepsilon, r}}\mathcal{H}^{n-m}\Big(P_{L^{\perp}}\big(S\cap B(x, \varepsilon\alpha\kappa r)\big)\Big)\omega_{m}r^m(\varepsilon\alpha\kappa+\beta)^m\\&\leq (\alpha+\beta)^m\sum_{x\in I_{\varepsilon, r}}\mathcal{H}^{n-m}\Big(P_{L^{\perp}}\big(S\cap B(x, \varepsilon\alpha\kappa r)\big)\Big)\omega_{m}r^m\\&\leq\frac{\varepsilon^{1/(k+m-n+1)}\omega_{m} r^{n-k} (\alpha+\beta)^m 3^n}{\gamma\alpha^{k+m-n}}.
            \end{split}
        \end{equation*}
Finally, we have $$S\oplus rC\subseteq (S_{\varepsilon, r}\oplus rC) \cup\big((S\setminus S_{\varepsilon, r})\oplus rC\big)\subseteq (S_{\varepsilon, r}\oplus rC) \cup\Big(\bigcup_{i=1}^M S_i\oplus\kappa r C^{\varepsilon}\oplus rC^{\varepsilon}\Big),$$ whence
\begin{equation*}
            \begin{split}
                \mathcal{M}_C^k(S)^*&\leq \limsup_{r\to0_+}\frac{\lambda^n(S_{\varepsilon, r}\oplus rC)}{\omega_{n-k}r^{n-k}}+\limsup_{r\to0_+}\frac{1}{\omega_{n-k}r^{n-k}}\lambda^n\bigg(\Big(\bigcup_{i=1}^M S_i\Big)\oplus (1+\kappa )rC^{\varepsilon}\bigg)\\&\leq\limsup_{r\to0_+}\frac{\varepsilon^{1/(k+m-n+1)}\omega_{m} (\alpha+\beta)^m 3^n}{\gamma\alpha^{k+m-n}\omega_{n-k}}+(1+\kappa)^{n-k}\sum_{i=1}^M\mathcal{M}^k_{C^{\varepsilon}}(S_i)\\&\leq\frac{\varepsilon^{1/(k+m-n+1)}\omega_{m} (\alpha+\beta)^m 3^n}{\gamma\alpha^{k+m-n}\omega_{n-k}}+(1+\kappa )^{n-k}\sum_{i=1}^{\infty}\mathcal{M}^k_{C^{\varepsilon}}(S_i)\\&=\frac{\varepsilon^{1/(k+m-n+1)}\omega_{m} (\alpha+\beta)^m 3^n}{\gamma\alpha^{k+m-n}\omega_{n-k}}\\&\quad\quad+(1+\kappa )^{n-k}\frac{1}{\omega_{n-k}}\int\limits_{S}\mathcal{H}^{n-k}(C^{\varepsilon}_x)\dH{k}(x).
            \end{split}
        \end{equation*}
        Using Lemma \ref{continuity}, we get the desired inequality $$\mathcal{M}_C^k(S)^*\leq\frac{1}{\omega_{n-k}}\int\limits_{S}\mathcal{H}^{n-k}(C_x)\dH{k}(x),$$ which completes the proof for those convex bodies $C$ with $\LO(C)=L$.

        If $\LO(C)\subsetneq L$, we consider $$C_{r}\coloneq C\oplus B_{L}(0, r),\quad r>0.$$ Using the first part of the proof, we conclude that $$\Phi_S(C)\leq\mathcal{M}^k_C(S)_*\leq\mathcal{M}^k_C(S)^*\leq\mathcal{M}^k_{C_{r}}(S)^*\leq\Phi_S(C_{r}),\quad r>0.$$ 
        Lemma \ref{continuity} gives $$\mathcal{M}^k_C(S)=\frac{1}{\omega_{n-k}}\int\limits_{S}\mathcal{H}^{n-k}(C_x)\dH{k}(x),$$ which completes the proof.
        \end{proof}\end{theorem}
\subsection{The Case $k+m-n=0$}
In this subsection, we show that the limit $\mathcal{M}^k_C(S)$ always exists and coincides with the quantity $$\frac{\mathcal{H}^{n-k}(C)}{\omega_{n-k}}\int\limits_{L^{\perp}}\mathcal{H}^{0}\big(S\cap(x+L)\big)\dH{k}(x)$$ provided that $k+m=n$. \begin{theorem}\label{k+m=n}
    Let $S\subseteq\mathbb{R}^n$ be a Borel set, and let $C\in\mathcal{C}^{n, m}$. If $k+m=n$, then $$\mathcal{M}^k_C(S)=\frac{\mathcal{H}^{n-k}(C)}{\omega_{n-k}}\int\limits_{L^{\perp}}\mathcal{H}^{0}\big(S\cap(x+L)\big)\dH{k}(x),$$ where $L=\LO(C).$
    \begin{proof}
        Fubini's theorem yields $$\mathcal{M}^k_{r, C}(S)=\frac{1}{\omega_{n-k}r^{n-k}}\int\limits_{L^{\perp}}\mathcal{H}^{n-k}\Big(\big(S\cap (x+L)\big)\oplus rC\Big)\dH{k}(x).$$

        Fatou's lemma and \eqref{0n} give \begin{equation*}
            \begin{split}
                \mathcal{M}^k_{C}(S)_*&\geq\int\limits_{L^{\perp}}\liminf_{r\to0_+}\frac{1}{\omega_{n-k}r^{n-k}}\mathcal{H}^{n-k}\Big(\big(S\cap (x+L)\big)\oplus rC\Big)\dH{k}(x)\\&=\frac{\mathcal{H}^{n-k}(C)}{\omega_{n-k}}\int\limits_{L^{\perp}}\mathcal{H}^{0}\big(S\cap (x+L)\big)\dH{k}(x).
            \end{split}
        \end{equation*} We see that if the function $$x\mapsto \mathcal{H}^{0}\big(S\cap (x+L)\big), \qquad x\in L^{\perp},$$ is not integrable, then $$\mathcal{M}^k_{C}(S)=\frac{\mathcal{H}^{n-k}(C)}{\omega_{n-k}}\int\limits_{L^{\perp}}\mathcal{H}^{0}\big(S\cap (x+L)\big)\dH{k}(x)=\infty.$$ On the other hand, if it is integrable, we can use the reverse Fatou's lemma and \eqref{0n} to obtain \begin{equation*}
            \begin{split}
                \mathcal{M}^k_{C}(S)^*&\leq\int\limits_{L^{\perp}}\limsup_{r\to0_+}\frac{1}{\omega_{n-k}r^{n-k}}\mathcal{H}^{n-k}\Big(\big(S\cap (x+L)\big)\oplus rC\Big)\dH{k}(x)\\&=\frac{\mathcal{H}^{n-k}(C)}{\omega_{n-k}}\int\limits_{L^{\perp}}\mathcal{H}^{0}\big(S\cap (x+L)\big)\dH{k}(x),
            \end{split}
        \end{equation*} since $$\mathcal{H}^{n-k}\Big(\big(S\cap (x+L)\big)\oplus rC\Big)\leq \mathcal{H}^{n-k}(rC)\mathcal{H}^0\big(S\cap (x+L)\big).$$ We get $$\mathcal{M}^k_{C}(S)=\frac{\mathcal{H}^{n-k}(C)}{\omega_{n-k}}\int\limits_{L^{\perp}}\mathcal{H}^{0}\big(S\cap (x+L)\big)\dH{k}(x),$$ as claimed.
    \end{proof}
\end{theorem}
\begin{corollary}
    Let $S\subseteq\mathbb{R}^n$ be a countably $\mathcal{H}^k$-rectifiable Borel set, and let $C\in\mathcal{C}^{n, m}$. If $k+m=n$, then $$\mathcal{M}^k_C(S)=\frac{1}{\omega_{n-k}}\int\limits_{S}\mathcal{H}^{n-k}(C_x)\dH{k}(x).$$
    \begin{proof} For $\mathcal{H}^k$-a.a. $x\in S$ we have the equality $$\mathcal{H}^{n-k}(C)J^{n-k}(P_{\textup{N}^{n-k}_xS}|_L)=\mathcal{H}^{n-k}(C_x).$$
        Theorem \ref{k+m=n} implies \begin{equation*}
            \begin{split}
                \mathcal{M}^k_C(S)&=\frac{\mathcal{H}^{n-k}(C)}{\omega_{n-k}}\int\limits_{L^{\perp}}\mathcal{H}^{0}\big(S\cap (x+L)\big)\dH{k}(x)\\&=\frac{\mathcal{H}^{n-k}(C)}{\omega_{n-k}}\int\limits_{S}J^k(P_{L^{\perp}}|_{\textup{T}^{k}_xS})\dH{k}(x)\\&=\frac{\mathcal{H}^{n-k}(C)}{\omega_{n-k}}\int\limits_{S}J^{n-k}(P_{\textup{N}^{n-k}_xS}|_L)\dH{k}(x)\\&=\frac{1}{\omega_{n-k}}\int\limits_{S}\mathcal{H}^{n-k}(C_x)\dH{k}(x),
            \end{split}
        \end{equation*} which completes the proof.
    \end{proof}
\end{corollary}
\subsection{The Case $k+m-n<0$}
We already know from the proof of Proposition \ref{LB} that if the set $S$ is countably $\mathcal{H}^k$-rectifiable, and $C\in\mathcal{C}^{n, m}$, where $k+m<n$, then $$\frac{1}{\omega_{n-k}}\int\limits_{S}\mathcal{H}^{n-k}(C_x)\dH{k}(x)=0.$$ 
The following theorem shows that the limit $\mathcal{M}_{C}^k(S)$ exists and equals $0$ as well.
\begin{theorem}
Let $S$ be a countably $\mathcal{H}^k$-rectifiable set in $\mathbb{R}^n$ and let $m\in\mathbb{N}_0$ be such that $k+m<n$. Then for every $C\in\mathcal{C}^{n, m}$, we have $$\mathcal{M}^{k}_C(S)=0.$$
    \begin{proof}
        Let us express $S$ as follows: $$S=S_0\cup\bigcup_{i=1}^{\infty}S_i,$$ where $S_0$ is $\mathcal{H}^k$-negligible and for every $i\in\mathbb N$, $S_i$ is a compact subset of a Lipschitz $k$-graph. We have $$S\oplus rC\subseteq \bigcup_{i=0}^{\infty}(S_i\oplus rC).$$

For each $i\in\mathbb N$, the set $S_i$ is a compact subset of a Lipschitz
$k$-graph, and $C$ is an $m$-dimensional convex body. Hence
$S_i\times C$ is contained in the Lipschitz image of a bounded subset of
$\mathbb R^{k+m}$, and therefore
\[
\mathcal H^{k+m}(S_i\times C)<\infty.
\]
For $i=0$, since $\mathcal H^k(S_0)=0$, the standard product estimate gives
\[
\mathcal H^{k+m}(S_0\times C)=0.
\] 

        Let $r>0$ and define $f_{r}\colon S\times C\to\mathbb{R}^n$ by $$f_{r}(x, c)\coloneq x+rc.$$ Then $f_r$ is Lipschitz and for every $i\in\mathbb{N}_0$, we have \begin{equation*}\begin{split}\mathcal{H}^{k+m}(S_i\oplus rC)&=\mathcal{H}^{k+m}\big(f_r(S_i\times C)\big)\\&\leq\textup{Lip}(f_r)^{k+m}\mathcal{H}^{k+m}(S_i\times C)\\&<\infty.\end{split}\end{equation*}

        Altogether, we get $$\lambda^n(S_i\oplus rC)=0, \qquad r>0, \, i\in\mathbb{N}_0,$$
       which gives us $$\lambda^{n}(S\oplus rC)\leq\sum_{i\in\mathbb{N}_0}\lambda^n(S_i\oplus rC)=0, \quad r>0.$$ Hence, the limit $\mathcal{M}_{C}^k(S)$ exists and coincides with $$\frac{1}{\omega_{n-k}}\int\limits_{S}\mathcal{H}^{n-k}(C_x)\dH{k}(x)=0,$$ which completes the proof.
    \end{proof}
\end{theorem}
\section{Dependence on the Choice of $C$}\label{secFurther}
In this section, we prove that if the $k$-dimensional $C$-anisotropic Minkowski content of a countably $\mathcal{H}^k$-rectifiable set $S$ exists, where $C\in\mathcal{C}^{n, n}$, and coincides with $$\frac{1}{\omega_{n-k}}\int\limits_{S}\mathcal{H}^{n-k}(C_x)\dH{k}(x),$$ then the same holds for any $C^{\prime}\in\mathcal{C}^{n, n}.$ 

Our proof adapts the argument used in \cite[Theorem 3.1]{frys}. We note, however, that \cite[Theorem 3.1]{frys} deals with \emph{localized} anisotropic Minkowski contents, while here we consider only the \emph{non-localized} case.
\begin{theorem}
\label{theoremExistence}
Let $S$ be a countably $\mathcal{H}^k$-rectifiable set, where $k\in\{1, \dots, n-1\}$, and let $C, C^{\prime}\in\mathcal{C}^{n, n}$. Then
\[
\mathcal{M}^k_C(S)=\frac{1}{\omega_{n-k}}\int\limits_{S}\mathcal{H}^{n-k}(C_x)\dH{k}(x)
\]
if and only if
\[
\mathcal{M}^k_{C^{\prime}}(S)=\frac{1}{\omega_{n-k}}\int\limits_{S}\mathcal{H}^{n-k}(C^{\prime}_x)\dH{k}(x). 
\]
\begin{proof}
Suppose that
\[
\mathcal{M}^k_C(S)=\frac{1}{\omega_{n-k}}\int\limits_{S}\mathcal{H}^{n-k}(C_x)\dH{k}(x)\qquad \big(=\Phi_S(C)\big).
\]
If $\mathcal{H}^k(S)=0$, then by Remark \ref{0inf2}, we get  $\Phi_S(C)=0$, and hence also $\mathcal{M}^k_C(S)=0$. Due to Remarks \ref{0inf} and \ref{0inf2}, we know that $\mathcal{M}^k_{C^{\prime}}(S)=\Phi_S(C^{\prime})=0$. The same strategy applies if $\mathcal{H}^k(S)=\infty$. Hence, we can assume that $0<\mathcal{H}^k(S)<\infty.$

Due to Proposition \ref{LB}, it suffices to show
\begin{align}\label{nikl}
\mathcal{M}^k_{C^{\prime}}(S)^*\leq\frac{1}{\omega_{n-k}}\int\limits_{S}\mathcal{H}^{n-k}(C^{\prime}_x)\dH{k}(x).
\end{align}

Since $S$ is countably $\mathcal{H}^{k}$-rectifiable, there exists a sequence $\{S_i\}_{i=1}^{\infty}$ of pairwise disjoint compact subsets of $S$ that covers $S$ up to an $\mathcal{H}^{k}$-negligible set, each of which is compact and contained in a $\mathcal C^1$ $k$-graph. 
Since $\mathcal{H}^k(S)>0$, we may assume that $\mathcal{H}^k(S_i)>0$ for each $i\in\mathbb{N}.$ Indeed, any set \(S_i\) with \(\mathcal H^k(S_i)=0\) may simply be
discarded and included in the negligible residual set. If only finitely
many sets remain, one of them may be decomposed, up to an
\(\mathcal H^k\)-negligible set, into countably many pairwise disjoint
compact subsets of positive \(\mathcal H^k\)-measure. This follows from
the nonatomicity and inner regularity of
\(\mathcal H^k\big|_{S_i}\). After reindexing, if necessary, the above assumption holds.

Let us prove (\ref{nikl}). Observe
\begin{equation*}
\begin{split}
\mathcal{M}^k_C(S)&=\frac{1}{\omega_{n-k}}\int\limits_{S}\mathcal{H}^{n-k}(C_x)\dH{k}(x)\\
&=\sum_{i\in\mathbb{N}}\frac{1}{\omega_{n-k}}\int\limits_{S_i}\mathcal{H}^{n-k}(C_x)\dH{k}(x)=\sum_{i\in\mathbb{N}}\mathcal{M}^k_C(S_i),
\end{split}
\end{equation*}
where the last equality follows from Lemma \ref{krectif}. For a given $\varepsilon>0$, there exists $K\in\mathbb{N}$ such that
\[
\mathcal{M}^k_C(S)-\sum_{i=1}^K\mathcal{M}^k_C(S_i)<\varepsilon^{k+1}.
\]

Since $S_1, \dots, S_K$ are pairwise disjoint nonempty compact sets, their mutual
distances are positive. Hence, for all sufficiently small $r>0$, the sets
$S_i\oplus rC$, $i=1, \dots, K$, are pairwise disjoint. Hence
\begin{equation*}
\sum_{i=1}^K\mathcal{M}_C^k(S_i)=\lim_{r\to0_+}\frac{1}{\omega_{n-k}r^{n-k}}\lambda^n\Big(\bigcup_{i=1}^K (S_i\oplus rC)\Big),
\end{equation*}
so that
\begin{equation}\label{rum}
\lim_{r\to0_+}\frac{1}{\omega_{n-k}r^{n-k}}\lambda^n\bigg((S\oplus rC)\setminus\Big(\bigcup_{i=1}^K (S_i\oplus rC)\Big)\bigg)<\varepsilon^{k+1}.
\end{equation}

Let positive constants $a$, $b$ and $c$ satisfy $B(0,a)\subseteq C$ and $B(0, c)\subseteq C^{\prime}\subseteq B(0,b)$, and define
\[
S_{\varepsilon, r}\coloneq \Big\{x\in S : \dist\Big(x, \bigcup_{i=1}^K S_i\Big)>2\diam(C)r\varepsilon \Big\}, \quad r>0.
\]

By Lemma \ref{bes}, there exists an at most countable set $I_{\varepsilon, r}\subseteq S_{\varepsilon, r}$ such that
\begin{align}\label{pp12}
    S_{\varepsilon, r}\subseteq \bigcup_{x\in I_{\varepsilon, r}} B(x, ar\varepsilon)&&\text{and}&&\sum_{x\in I_{\varepsilon, r}}\chi_{B(x, ar\varepsilon)}\leq 3^n.
\end{align}
We have \begin{equation*}
    \begin{split}
        \mathcal{H}^0(I_{\varepsilon, r})\omega_n (ar\varepsilon)^n&=\sum_{x\in I_{\varepsilon, r}}\omega_n (ar\varepsilon)^n\\&=\sum_{x\in I_{\varepsilon, r}}\lambda^n\big(B(x, ar\varepsilon)\big)\\&\leq 3^n\lambda^n\Big(\bigcup_{x\in I_{\varepsilon, r}}B(x, ar\varepsilon)\Big).
    \end{split}
\end{equation*}

Notice
\[
\bigcup_{x\in I_{\varepsilon, r}} B(x, ar\varepsilon)\subseteq S_{\varepsilon, r}\oplus B(0, ar\varepsilon) \subseteq (S\oplus r\varepsilon C) \setminus \Big(\bigcup_{i=1}^K (S_i\oplus r\varepsilon C)\Big).
\]

Inequality \eqref{rum} implies
\begin{equation}\label{limsup}\limsup_{r\to0_+}\mathcal{H}^0(I_{\varepsilon, r}) r^{k}\leq \frac{\omega_{n-k}3^n\varepsilon}{a^n\omega_n}.
\end{equation}
We have $$S\subseteq (S\setminus S_{\varepsilon, r})\cup S_{\varepsilon, r}.$$

On the one hand,
\[
S\setminus S_{\varepsilon, r}=\Big\{x\in S : \dist\Big(x, \bigcup_{i=1} ^K S_i\Big)\leq 2\diam(C)\varepsilon r\Big\},
\]
from which we see that $$S\setminus S_{\varepsilon, r}\subseteq\bigcup_{i=1}^K S_i\oplus B\big(0, 3\diam(C)\varepsilon r\big),$$ whence \begin{equation}\label{limsup2}
    \begin{split}
        \lambda^n\big((S\setminus S_{\varepsilon, r})\oplus rC^{\prime}\big)\leq \lambda^n\Big(\bigcup_{i=1}^K S_i\oplus r\big(1+c^{-1}3\diam(C)\varepsilon\big)C^{\prime}\Big)
    \end{split}
\end{equation}

On the other hand, \newcommand\myeqfkl{\mathrel{\stackrel{\makebox[0pt]{\mbox{\normalfont\tiny (\ref{pp12})}}}{\subseteq}}}
\[
S_{\varepsilon, r} \oplus rC^{\prime}\myeqfkl \bigcup_{x\in I_{\varepsilon, r}} \big(B(x, ar\varepsilon)\oplus B(0, rb)\big)=\bigcup_{x\in I_{\varepsilon, r}} B\big(x, (a\varepsilon+b)r\big).
\]
Hence, using \eqref{limsup}, we obtain
\begin{equation}\label{kvak}
\begin{split}
\limsup_{r\to0_+}\frac{\lambda^n(S_{\varepsilon, r}\oplus rC^{\prime})}{\omega_{n-k}r^{n-k}}
&\leq \limsup_{r\to0_+}\frac{1}{\omega_{n-k}r^{n-k}}\lambda^n\Big(\bigcup_{x\in I_{\varepsilon, r}} B\big(x, (a\varepsilon+b)r\big)\Big)\\&\leq\frac{\omega_n(a\varepsilon+b)^n}{\omega_{n-k}}\limsup_{r\to0_+}\mathcal{H}^0(I_{\varepsilon, r})r^k\\&\leq\frac{(a\varepsilon+b)^n3^n\varepsilon}{a^n}.
\end{split}
\end{equation}
Finally, using \eqref{limsup2}, \eqref{kvak} and Lemma \ref{krectif}, we conclude that
\[
\begin{split}
\mathcal{M}^k_{C^{\prime}}(S)^* &\leq \limsup_{r\to0_+}\frac{\lambda^n(S_{\varepsilon, r}\oplus rC^{\prime})}{\omega_{n-k}r^{n-k}} + \limsup_{r\to0_+}\frac{\lambda^n\big((S\setminus S_{\varepsilon, r})\oplus rC^{\prime}\big)}{\omega_{n-k}r^{n-k}}\\
&\leq \frac{(a\varepsilon+b)^n3^n\varepsilon}{a^n}+ \limsup_{r\to0_+}\frac{1}{\omega_{n-k}r^{n-k}}\lambda^n\Big(\bigcup_{i=1}^K S_i\oplus r\big(1+c^{-1}3\diam(C)\varepsilon\big)C^{\prime}\Big)\\
&= \frac{(a\varepsilon+b)^n3^n\varepsilon}{a^n}+\big(1+c^{-1}3\diam(C)\varepsilon\big)^{n-k}\sum_{i=1}^K\mathcal{M}^k_{C^{\prime}}(S_i)\\&\leq \frac{(a\varepsilon+b)^n3^n\varepsilon}{a^n}+\big(1+c^{-1}3\diam(C)\varepsilon\big)^{n-k}\sum_{i=1}^{\infty}\mathcal{M}^k_{C^{\prime}}(S_i)\\&\leq \frac{(a\varepsilon+b)^n3^n\varepsilon}{a^n}+\big(1+c^{-1}3\diam(C)\varepsilon\big)^{n-k}\frac{1}{\omega_{n-k}}\int\limits_{S}\mathcal{H}^{n-k}(C^{\prime}_x)\dH{k}(x).
\end{split}
\]

Since $\varepsilon>0$ was arbitrary, we conclude
\[
\mathcal{M}^k_{C^{\prime}}(S)^* \leq \frac{1}{\omega_{n-k}}\int\limits_{S}\mathcal{H}^{n-k}(C^{\prime}_x)\dH{k}(x).
\]

The converse implication follows by interchanging the roles of \(C\) and \(C'\).
\end{proof}
\end{theorem}

\begin{rem}
Theorem~\ref{theoremExistence} implies, in particular, Theorem~\ref{sq}. Indeed, the validity of the AFP-$k$-condition ensures that
\[
\mathcal{M}^k_{B(0, 1)}(S) = \mathcal{H}^k(S)
\]
as shown in \cite[Theorem~2.104]{ambrosio} and $\mathcal{H}^k(S)<\infty$ (see Remark \ref{pp15}). Hence, by Theorem~\ref{theoremExistence},
\[
\mathcal{M}^k_{C}(S)
= \frac{1}{\omega_{n-k}}\int\limits_{S}\mathcal{H}^{n-k}(C_x)\dH{k}(x)
\]
for every $C \in \mathcal{C}^{n,n}$. By the standard approximation argument, we extend this to arbitrary convex body $C\in\mathcal{C}^n.$
\end{rem}

\end{document}